\documentclass[11pt]{amsart}

\usepackage{amsfonts,amsthm,amsmath,enumerate,amssymb,latexsym,color,tcolorbox,tikz,mathrsfs,bm,subfig,tcolorbox}
\usetikzlibrary{patterns,patterns.meta,bending,angles,quotes,shapes.geometric}
\usepackage[text={6in,9in},centering]{geometry}
\usepackage[backref=page,
            colorlinks,
            linkcolor=teal,
            anchorcolor=red,
            citecolor=red]{hyperref}

\theoremstyle{plain}
\newtheorem{theorem}{Theorem}[section]
\newtheorem{lemma}[theorem]{Lemma}
\newtheorem{corollary}[theorem]{Corollary}
\newtheorem{proposition}[theorem]{Proposition}

\theoremstyle{definition}
\newtheorem{definition}[theorem]{Definition}
\newtheorem{example}[theorem]{Example}

\theoremstyle{remark}
\newtheorem{remark}[theorem]{Remark}

\numberwithin{equation}{section}

\DeclareMathOperator{\dimh}{dim_H}
\DeclareMathOperator{\dist}{dist}

\newcommand{\red}{\color{red}}

\def\R{\mathbb R}
\def\Z{\mathbb Z}
\def\D{\mathcal D}

\def\j{\bm j}

\def\Q{\mathbb Q}
\def\h{\mathcal H}
\def\S{\mathcal S}

\date{\today}

\begin{document}

\title[Self-embeddings of Bedford-McMullen carpets with dependent ratios]{Self-embedding similitudes of Bedford-McMullen carpets with dependent ratios}

\author{Jian-Ci Xiao}
\address{School of Mathematics, Nanjing University of Aeronautics and Astronautics, Nanjing, China}
\email{jcxiao@nuaa.edu.cn}


\begin{abstract}
    We prove that any non-degenerate Bedford-McMullen carpet  does not admit oblique self-embedding similitudes; that is, if $f$ is a similitude sending the carpet into itself, then the image of the $x$-axis under $f$ must be parallel to one of the principal axes. This result leads to a logarithmic commensurability result on the contraction ratios of such embeddings, completing a previous study by Algom and Hochman [Ergod. Th. \& Dynam. Sys. {\bf 39} (2019), 577--603] on Bedford-McMullen carpets generated by multiplicatively independent exponents. Our approach also provides a new proof of their non-obliqueness statement that avoids analyzing the tangent sets.

    For the self-similar case, however, we construct a generalized Sierpi\'nski carpet that is symmetric with respect to an appropriate oblique line and hence admits a reflectional oblique self-embedding. As a complement, we prove that if a generalized Sierpi\'nski carpet satisfies the strong separation condition and permits an oblique rotational self-embedding similitude, then the tangent of the rotation angle takes values $\pm 1$.
\end{abstract}

\maketitle

\section{Introduction}

One way to gain a better comprehension of the geometric structure of fractal sets is to delve into the study of their self-embedding mappings, especially affine mappings or even similitudes. Note that a similitude $f$ on $\R^d$ can be written as $f(x)=\lambda Ox+a$, where $\lambda\geq 0$ is the \emph{similarity ratio} (or \emph{contraction ratio} if $\lambda<1$), $O$ is an orthogonal transformation on $\R^d$ (the \emph{orthogonal part}) and $a\in\R^d$ (the \emph{translation part}). For convenience, we will refer to $f$ a \emph{self-embedding similitude} of a set $K\subset\R^d$ if $f(K)\subset K$. A general principle is that imposing additional restrictions on $K$ usually results in increased constraints on its self-embedding similitudes.

Research in this direction traces back to Furstenberg~\cite{Fur67} who showed that if $A\subset \mathbb{T}:=\R/\Z$ is an infinite closed set invariant under the operations $\times 2 \pmod 1$ and $\times 3 \pmod 1$, then necessarily $A=\mathbb{T}$. A crucial reason is the multiplicative independence of $2$ and $3$, that is, neither is a rational power of the other. More recently, such phenomena have been greatly extended to the self-similar setting. Recall that a compact set $K\subset\R^d$ is called \emph{self-similar} if $K=\bigcup_{i=1}^p \varphi_i(K)$, where $\{\varphi_i\}_{i=1}^p$ is an iterated function system (IFS) consisting of contracting similitudes; see~\cite{Hut81}. In 2009, Feng and Wang~\cite{FW09} established a fundamental logarithmic commensurability theorem: if $K\subset\R$ is a totally disconnected \emph{homogeneous} self-similar set (meaning that all maps in the associated IFS share the same contraction ratio and orthogonal part) satisfying the open set condition, then any self-embedding similitude of $K$ must contract by a rational power of the common ratio. This result was later extended to higher dimensions and to inhomogeneous IFSs by Elekes, Keleti and M\'ath\'e~\cite{EKM10}, albeit with the open set condition being substituted by the strong separation condition. If in addition the IFS is homogeneous, the author~\cite{Xiao24} proved the relative openness of the embedded image. In~\cite{Alg201}, Algom characterized all affine self-embeddings of self-similar sets in the plane under the strong separation condition.

For self-affine sets, however, there are much fewer existing results. One reason for this disparity is that a natural rescaling approach, which works well in the self-similar case, usually fails in the self-affine case due to the inherent distortion. Nevertheless, the classic family of Bedford-McMullen carpets, with their simple lattice structure and delicate intrinsic properties, presents an optimal subject for investigation. A formal definition is as follows.

\begin{definition}
    Let $n\geq m\geq 2$ be integers and let $\Lambda\subset \{0,1,\ldots,n-1\}\times\{0,1,\ldots,m-1\}$ be non-empty. Set $\varphi_{i,j}$ to be the affine map given by 
    \begin{equation}\label{eq:varphiij}
        \varphi_{i,j}(x,y) = \Big( \frac{x+i}{n}, \frac{y+j}{m} \Big), \quad (i,j)\in\Lambda.
    \end{equation}
    The attractor $K(n,m,\Lambda)$ associated with the self-affine IFS $\{\varphi_{i,j}:(i,j)\in\Lambda\}$ is usually called a \emph{Bedford-McMullen carpet} when $n>m$, and a \emph{generalized Sierpi\'nski carpet} when $n=m$. More precisely, 
    \[
        K(n,m,\Lambda) = \Big\{ \Big( \sum_{k=1}^\infty \frac{i_k}{n^k}, \sum_{k=1}^\infty \frac{j_k}{m^k} \Big): (i_k,j_k)\in\Lambda \Big\}.
    \]
    To avoid trivial cases, in this paper we always assume that $1<\#\Lambda<mn$, where $\#$ denotes the cardinality. 
\end{definition}

In~\cite{EKM10}, Elekes, Keleti and M\'ath\'e studied the intersections of Bedford-McMullen carpets with their translated copies. They observed a measure drop phenomenon occurring in these intersections, except in trivial cases. A systematic study on self-embedding similitudes of Bedford-McMullen carpets was conducted by Algom and Hochman~\cite{AH19}, who showed that if the horizontal and vertical ratios $n,m$ are multiplicatively independent, then the carpet permits only nearly trivial self-embedding similitudes. To state and extend their findings, let us introduce a convenient terminology.

\begin{definition}[Obliqueness]
    \,
    \begin{enumerate}
        \item A line $L\subset\R^2$ is called \emph{oblique} if it is not parallel to one of the two principle coordinate axes.
        \item A similitude $f$ on $\R^2$ with positive similarity ratio is called \emph{oblique} if $f$ sends the $x$-axis to an oblique line.
    \end{enumerate}
\end{definition}

Algom and Hochman's main result is the following theorem.

\begin{theorem}[\cite{AH19}]\label{thm:AH19}
    Let $K=K(n,m,\Lambda)$ be a Bedford-McMullen carpet with $\frac{\log n}{\log m}\notin\Q$. Suppose $K$ is not supported on any line and is not a Cartesian product of the unit interval $[0,1]$ and some Cantor set. If $f$ is a self-embedding similitude of $K$, then $f$ is not oblique and has similarity ratio $1$.
\end{theorem}

To achieve this, Algom and Hochman developed an advanced tangent set machinery: they zoom in on every point inside the carpet using $m$-adic rescaling, describe the resulting tangent sets as unions of fibered products, and analyze how similitude embeddings transform the tangent sets. This microscopic analysis is coupled with a projection theorem of Peres and Shmerkin~\cite{PS09} for products of deleted-digit sets (see also Lemma~\ref{lem:projofcarpet}), which requires multiplicative independence of $m$ and $n$ to ensure that every non-principal projection has maximal dimension. Such maximality creates a dimension mismatch whenever the orthogonal part of the embedding similitude is oblique, thereby ruling out non-diagonal linear parts. For the similarity ratio, the non-obliqueness statement and the structure of tangent sets reduce the problem to that of deleted-digit sets. Since it is well known that a deleted-digit set in base $m$ cannot be embedded into another one in base $n$ when $\frac{\log m}{\log n}\notin\Q$ (see~\cite{FHR14}), this phenomenon heuristically forces the contraction ratio of the embedding similitude to be simultaneously dependent on both $n$ and $m$, which is only possible when the similitude is an isometry. The non-degenerate and non-product assumptions are also natural. For example, if $K$ is a horizontal line segment or a self-similar carpet (e.g., the Cartesian product of the middle-third Cantor set and the unit interval), then $f$ clearly need not to be an isometry. 

Our main purpose is to study to what extent Theorem~\ref{thm:AH19} remains true for the dependent case (i.e., $\frac{\log n}{\log m}\in\Q$) or even the self-similar case (i.e., $n=m$). The presence of dependence inherently complicates the self-embedding problem. Fortunately, first we are able to extend the non-obliqueness statement to the dependent case. Our method accommodates both dependent and independent cases and provides a new and more direct proof for the independent case without analyzing the tangent sets as in~\cite{AH19}.

\begin{theorem}\label{thm:main1}
    Let $K=K(n,m,\Lambda)$ be a Bedford-McMullen carpet with $n>m$ that is not supported on any line. If $f$ is a self-embedding similitude of $K$, then $f$ is not oblique.
\end{theorem}

We remark that the non-obliqueness statement is closely related to a slicing problem for Bedford-McMullen carpets. For example, writing $N$ to be the maximal number of rectangles selected in any row of the initial pattern (see Section~\ref{sec:pre} for a more precise definition), it is easy to find some $0\leq y\leq 1$ such that $\dimh (K\cap (\R\times\{y\}))=\frac{\log N}{\log n}$, where $\dimh$ denotes the Hausdorff dimension. But an oblique self-embedding similitude of $K$ sends this horizontal slice into an oblique one, and if it could be shown that any oblique slice of $K$ has Hausdorff dimension strictly less than $\frac{\log N}{\log n}$, then the existence of oblique self-embeddings would be ruled out. Thus a natural question is: under what conditions does such an upper bound for oblique slices hold? By Marstrand's slicing theorem~\cite{Mar54}, for any oblique line $L$ in $\R^2$, 
\begin{equation}\label{eq:marstrand}
    \dimh (K\cap (L+a)) \leq \max\{\dimh K-1,0\} \quad \text{for a.e. } a\in L^\perp.
\end{equation}
In the case when $K$ has non-uniform horizontal fibres (that is, there are two rows containing different numbers of selected rectangles), a straightforward calculation using the formula for $\dimh K$ (see \eqref{eq:dimhk}) yields
\[
    \dimh K-1 < \frac{1}{\log m}\cdot\log \Big( m\cdot N^\frac{\log m}{\log n} \Big)-1 = \frac{\log N}{\log n}.
\]
So almost all oblique slices of $K$ have Hausdorff dimension strictly less than $\frac{\log N}{\log n}$.
When $\frac{\log n}{\log m}\notin\Q$, it is possible that the upper bound specified in~\eqref{eq:marstrand} holds for all oblique slices rather than merely almost all, which would lead to a quick proof for certain special cases of Theorem~\ref{thm:main1}. This slicing problem is formally stated in the nice survey~\cite{Fra21} and appears to be quite challenging. Recent progress can be found in~\cite{Alg20,AW23}. However, in the dependent case the upper bound in~\eqref{eq:marstrand} may not hold, as indicated by~\cite{BR14}. Moreover, even when the bound holds for Hausdorff dimension, other notions of dimension (such as the Assouad dimension) may behave differently, a point emphasized in~\cite{AW23}.

On the other hand, one cannot hope to extend the isometry statement in Theorem~\ref{thm:AH19} to the dependent case. For example, when $n=m^2$, it is easy to construct a Bedford-McMullen carpet which is actually a self-similar carpet. In such instances, there exist numerous non-isometric self-embedding similitudes. Nevertheless, a logarithmic commensurability type conclusion as follows still holds.

\begin{theorem}\label{thm:main2}
    Let $K=K(n,m,\Lambda)$ be a Bedford-McMullen carpet with $n>m, \frac{\log n}{\log m}\in\Q$ and let $f$ be a self-embedding similitude of $K$. If $K$ is not contained in any line, then $\frac{\log\lambda}{\log n}\in\Q$, where $\lambda$ denotes the similarity ratio of $f$.
\end{theorem}

The statement was implicitly shown in the work of Algom and Hochman~\cite{AH19}, although their theorem is stated for the case where $n$ and $m$ are independent. Roughly speaking, once it is proven that any self-embedding of the carpet must be non-oblique, the structure of the tangent sets allows them to reduce the problem to that of deleted-digit sets. The desired commensurability then follows directly from the aforementioned theorem of Feng and Wang (see Lemma~\ref{lem:fw09}). In the present paper, we provide an alternative proof of this result that avoids analyzing the tangent sets (see Section~\ref{subsec:proofofthecom}).

Finally, we consider the self-similar case when $n=m$, which can be viewed as the ``most dependent'' case. While one might anticipate a non-oblique statement in this context, it is not hard to construct a generalized Sierpi\'nski carpet which is symmetric with respect to an appropriate oblique line, thus permitting an oblique reflectional self-embedding. See Example~\ref{exa:nonobsiercat}. So the non-oblique statement fails in the self-similar setting. However, considering rotational self-embeddings in lieu of reflectional ones and under the assumption of the strong separation condition, we indeed eliminate almost all possibilities. Here a generalized Sierpi\'nski carpet $K=K(n,\Lambda)$ is said to satisfy the \emph{strong separation condition} if elements in $\{\varphi_{i,j}(K)\}_{(i,j)\in\Lambda}$ are mutually disjoint. 

\begin{theorem}\label{thm:main3}
    Let $K=K(n,\Lambda)$ be a generalized Sierpi\'nski carpet satisfying the strong separation condition. Let $\theta\in\R$ and let $R_\theta$ be the counterclockwise rotation at the origin by angle $\theta$. If $f(x)=\lambda R_\theta x+a$ is an oblique similitude sending $K$ into itself, then $|\tan\theta|=1$.
\end{theorem}

We outline here some rough ideas on how to prove the above results. In contrast to Algom and Hochman~\cite{AH19}, we work with the hierarchical grid structure behind the carpet directly and study the tangents of the embedding image rather
than of the carpet. One of the key tools we use is a new logarithmic commensurability theorem for non-stationary deleted-digit sets (where digit sets vary by level); see Proposition~\ref{prop:allnthenrational}. 

For Theorem~\ref{thm:main1}, we track how those $(m\times n)$-adic rectangles are transformed under an oblique self-embedding similitude of $K$. Assume that such an embedding exists. We first show that $K$ contains no oblique line segments, which reduces to the case when no row in the initial pattern is full. If every row is non-empty, a geometric flattening argument shows that the oblique image would have to enter a forbidden region, leading to a contradiction. In the remaining case, we construct another oblique similitude that embeds the carpet to a fibred product set $C\times\R$. Using the aforementioned commensurability result together with a careful analysis of boundary points with matching vertical coordinates, we show that the oblique image near such points would protrude horizontally beyond $C$, again yielding a contradiction.

For Theorem~\ref{thm:main2}, it suffices to consider homotheties. We show that such an embedding should send a scaled copy of some maximal horizontal slice into itself, and the desired conclusion follows directly from our refined version of the commensurability theorem.

Theorem~\ref{thm:main3} is achieved as follows. Note that the convex hull of any non-degenerate Sierpi\'nski carpet is a polygon. On the one hand, we show that if there exists an oblique rotational $f$, then the rotational angle is a rational multiple of $\pi$, and simultaneously can be expressed as the sum of finitely many inner angles (modulus $2\pi$) of the above polygon. On the other hand, we prove that all of these inner angles have rational tangents. So the rotation angle has a rational tangent, and the desired conclusion follows directly from a theorem of Niven (see Lemma~\ref{lem:niven}). Here the strong separation condition plays a crucial role in deriving the logarithmic commensurability of the contraction ratio of $f$ (by a result in~\cite{EKM10}). In fact, by employing some mass comparison arguments, the separation condition in Theorem~\ref{thm:main3} could be relaxed to the requirement of total disconnectedness; however, this does not present a substantial improvement and we will not explore it further within the scope of this paper. 

The paper is organized as follows. In the subsequent section, we introduce general notation and provide several useful observations on Bedford-McMullen carpets and horizontal slices. In Section~\ref{sec:logcom}, we establish a logarithmic commensurability theorem for non-stationary deleted-digit sets. Section~\ref{sec:indep} presents a new proof of the non-obliqueness statement for the independent case. Section~\ref{sec:dep} is devoted to the proofs of Theorems~\ref{thm:main1} and~\ref{thm:main2}. Finally, we establish Theorem~\ref{thm:main3} in Section~\ref{sec:selfsim}.

\section{Preliminaries}\label{sec:pre}

\subsection{Basic notation}
Throughout this paper, we write $\mathcal{S}(\R^d)$ to be the collection of all similitudes on $\R^d$ with positive similarity ratio. For a collection $\mathcal{A}$ of subsets of $\R^d$,  $\bigcup\mathcal{A}$ denotes the union of sets in $\mathcal{A}$. For $A\subset\R^d$ and $x\in\R^d$, we denote by $A+x$ (or $x+A$) the translated copy of $A$ by $x$, that is, $A+x:=\{a+x: a\in A\}$. We also write $|A|:=\sup\{|x-y|:x,y\in A\}$ as the diameter of $A$.

For any oblique linear line $L\subset\R^2$, $\pi_L$ denotes the orthogonal projection onto $L$. Meanwhile, we write $\pi_1$ and $\pi_2$ to be the orthogonal projections onto the $x$-axis and $y$-axis, respectively. The concatenation $fg$ of any two mappings $f,g$ simply means the composition $f\circ g$. 

For integers $n,m\geq 2$, we write $n\sim m$ if $\frac{\log n}{\log m}\in\Q$ and $n\nsim m$ otherwise. For each $k\geq 1$, intervals of the form $[\frac{i}{n^k},\frac{i+1}{n^k}]$ (resp. $[\frac{i}{m^k},\frac{i+1}{m^k}]$) will be referred to as \emph{$n^k$-adic intervals} (resp. \emph{$m^k$-adic intervals}), where $0\leq i\leq n^k-1$.

\subsection{Bedford-McMullen carpets: general observations}

A Bedford-McMullen carpet $K=K(n,m,\Lambda)$ can be constructed through a standard iterative process as follows. One first divides the unit square $[0,1]^2$ into an $n\times m$ grid, selecting a subset of rectangles formed by the grid (called the \emph{initial pattern}) and then repeatedly substituting the initial pattern on each of the selected rectangles. The limit set is just $K$. Recalling the notation in~\eqref{eq:varphiij}, for $k\geq 1$ we call every element in 
\[
    \{\varphi_{i_1,j_1}\circ\cdots\circ\varphi_{i_k,j_k}([0,1]^2): (i_1,j_1),\ldots,(i_k,j_k)\in\Lambda\}
\]
a \emph{level-$k$ rectangle}. For any level-$k$ rectangle $R$, we write $\varphi_R$ to be the natural affine map sending $[0,1]^2$ onto $R$.

Below we present two geometric observations regarding Bedford-McMullen carpets, which are significantly influenced by the condition $n>m$ and might be of independent interest.

\begin{proposition}\label{prop:projofkisinf}
    Let $K=K(n,m,\Lambda)$ be a Bedford-McMullen carpet with $n>m$. If there is a pair of selected rectangles in the initial pattern lying in different rows, then $\#\pi_L(K)=\infty$ for every oblique linear line $L$.
\end{proposition}
\begin{proof}
    Up to an affine transformation, it suffices to prove that 
    \[
        \#\{x\cos\theta+y\sin\theta: (x,y)\in K\} = \infty
    \]
    for every $\theta\in(0,\pi)\setminus\{\pi/2\}$. Note that $\cos\theta$ and $\sin\theta$ are both nonzero. Since there are two selected rectangles lying in different rows, we can find $(i_*,j_*),(i^*,j^*)\in\Lambda$ with $j_*\neq j^*$.

    Let $p\geq 1$. Pick a large $k$ so that $\frac{m^{k+p}}{n^k}<|\frac{\sin\theta}{\cos\theta}|$ (here we use $n>m$) and an arbitrary level-$k$ rectangle $R$. By definition, $\varphi_R$ is of the form $(x,y)\mapsto (\frac{x}{n^k},\frac{y}{m^k})+(c_1,c_2)$ for some $0\leq c_1,c_2\leq 1$. Fix any point $(x_0,y_0)\in K$ and consider the set
    \[
        E :=\{\varphi_{i_1,j_1}\cdots\varphi_{i_p,j_p}(x_0,y_0): (i_t,j_t)\in \{(i_*,j_*),(i^*,j^*)\} \text{ for } 1\leq t\leq p\}.
    \]
    Since the maps $\varphi_{i,j}$ are injective and the two choices of $\{(i_t, j_t)\}_t$ produce distinct compositions, $E$ consists of exactly $2^p$ distinct points.

    Now consider $\varphi_R(E)$. For any two distinct points in $\varphi_R(E)$, say 
    \[
        \left\{\begin{array}{ll} (a,b) = \varphi_R\varphi_{i_1,j_1}\cdots\varphi_{i_p,j_p}(x_0,y_0), \\ (a',b') = \varphi_R\varphi_{i'_1,j'_1}\cdots\varphi_{i'_p,j'_p}(x_0,y_0). \end{array}\right.
    \]
    we compute that 
    \begin{align*}
        |b-b'| &= m^{-k}\Big|\Big( m^{-p}y_0+ \sum_{t=1}^p j_tm^{-t} \Big) - \Big( m^{-p}y_0+ \sum_{t=1}^p j'_tm^{-t} \Big)\Big| \\
        &= m^{-k}\Big|\sum_{t=1}^p (j_t-j'_t)m^{-t} \Big|,
    \end{align*}
    which is at least $m^{-k-p}$ unless $(j_1,\ldots,j_p)=(j'_1,\ldots,j'_p)$. Since $j_*\neq j^*$, this occurs precisely when $(i_t,j_t)=(i'_t,j'_t)$ for $1\leq t\leq p$. Thus the $y$-coordinates of points in $\varphi_R(E)$ are $m^{-k-p}$-separated.
    
    It follows that $\{a\cos\theta+b\sin\theta:(a,b)\in\varphi_R(E)\}$ contains exactly $\# \varphi_R(E)=2^p$ real numbers: otherwise, there are distinct $(a,b),(a',b')\in\varphi_R(E)$ with $a\cos\theta+b\sin\theta=a'\cos\theta+b'\sin\theta$, which implies that 
    \begin{align*}
        \frac{|\sin\theta|}{|\cos\theta|} &= \frac{|a-a'|}{|b-b'|}  \\
        &\leq \frac{n^{-k}}{|b-b'|} && \text{(since $\varphi_R(E)\subset R$)} \\
        &\leq \frac{n^{-k}}{m^{-k-p}} && \text{(since $|b-b'|\geq m^{-k-p}$)} \\
        &= \frac{m^{k+p}}{n^k} < \frac{|\sin\theta|}{|\cos\theta|}, && \text{(by our choice of $k$)}
    \end{align*}
    a contradiction. 
    
    Since $(x_0,y_0)\in K$, $E\subset K$ and thus $\varphi_R(E)\subset K$. In particular, 
    \[
        \#\{x\cos\theta+y\sin\theta: (x,y)\in K\} \geq \#\{a\cos\theta+b\sin\theta:(a,b)\in\varphi_R(E)\} \geq 2^p.
    \]
    The proof is then completed because $p$ can be taken arbitrarily large.
\end{proof}

\begin{proposition}\label{prop:noobliquelines}
    Let $K=K(n,m,\Lambda)$ be a Bedford-McMullen carpet with $n>m$. Then $K$ does not contain any oblique line segment of positive length.
\end{proposition}

The geometric idea behind the proof is as follows. Since not every rectangle is selected in the initial pattern, the construction results in holes at every scale throughout the carpet. Consequently, attempting to put an oblique segment inside the carpet faces the challenge of anisotropic stretching---where the horizontal and vertical directions shrink at distinct rates. Over successive iterations, this stretching significantly flattens the segment, inevitably making it pass through one of those holes. It is worth noting that the statement fails when $n=m$. For example, when $\Lambda=\{(i,i):0\leq i\leq n-1\}$, the carpet is nothing but a diagonal of the unit square. 

Before proceeding to the proof, let us mention some related results. In the case when $n\nsim m$, Algom~\cite{Alg20} even found a constant $0<c<1$ (depending on $K$) such that every oblique slice of $K$ has Assouad dimension at most $1-c$. In particular, $K$ is unable to contain an entire oblique segment. On the other hand, as shown by Algom and Wu~\cite{AW25}, certain oblique slices can still exhibit significant dimensionality under some mild conditions.

\begin{proof}
    Since we always assume that $\#\Lambda<mn$, there are $0\leq i_0\leq n-1,0\leq j_0\leq m-1$ such that $(i_0,j_0)\notin\Lambda$. Let $Q:=(\frac{i_0}{n},\frac{i_0+1}{n})\times(\frac{j_0}{m},\frac{j_0+1}{m})$. Since $(i_0,j_0)\notin\Lambda$, $Q$ is the open rectangle corresponding to that missing cell and thus $Q\cap K=\varnothing$.
    
    Suppose $K$ contains an oblique line segment $\ell$ of length $|\ell|>0$. Let $u$ denote the slope of $\ell$; without loss of generality, we may assume that $u>0$. The goal is to show that this results in $Q\cap K\neq\varnothing$, leading to a contradiction.
    
    Fix an integer $k$ so large that $n^{-k}\leq \frac{|\pi_1(\ell)|}{3n}$ and $\frac{m^ku}{n^k}<\frac{1}{2}$. We claim that there exists a level-$k$ rectangle $R$ such that $\ell$ meets both of the left and right sides of $R$. To see this, let $(a_1,b_1)$ be the left endpoint of $\ell$. Define $i_*:=\min\{0\leq i\leq n^k: \frac{i}{n_k}\geq a_1\}$ and $y_0$ be such that $(\frac{i_*}{n^k},y_0)\in\ell$. Then define recursively that 
    \[
        y_t := y_{t-1}+\frac{u}{n^k}, \quad 1\leq t\leq \lfloor |\pi_1(\ell)|n^k\rfloor-1,
    \]
    where $\lfloor\cdot\rfloor$ denotes the integer part. So $(\frac{i_*+t}{n^k},y_t)\in\ell$ because $u$ is the slope of $\ell$. Since $|y_t-y_{t-1}|=\frac{u}{n^k}<\frac{1}{2m^k}$, there must exist some $t_0$ such that the interval $[y_{t_0-1},y_{t_0}]$ is entirely contained in an $m^k$-adic interval, say $[\frac{j'}{m^k},\frac{j'+1}{m^k}]$. In particular, 
    \[
        \frac{j'}{m^k} \leq y_{t_0-1}<y_{t_0}\leq\frac{j'+1}{m^k}.
    \]
    Taking 
    \[
        R = \Big[ \frac{i_*+t_0-1}{n^k},\frac{i_*+t_0}{n^k} \Big]\times \Big[ \frac{j'}{m^k},\frac{j'+1}{m^k} \Big],
    \]    
    we see that $\ell$ meets both of the left and right sides of $R$, as desired.
    
    Replacing $\ell$ by its intersection with $R$ if necessary, we may assume that $\ell\subset R$. Consequently, $\ell\subset\varphi_R(K)$ and hence $\varphi_R^{-1}(\ell)\subset K$. Note that $\varphi_{R}^{-1}(\ell)$ has slope $\frac{m^ku}{n^k}$. Let $(0,a)$ be the intersection point of the segment $\varphi_{R}^{-1}(\ell)$ with the left side of $\varphi_{R}^{-1}(R)=[0,1]^2$; then $\varphi_R^{-1}(\ell)$ meets the right side of $[0,1]^2$ at $(1,a+\frac{m^ku}{n^k})$. For convenience, we define 
    \[
        g(x):=\frac{m^ku}{n^k}x+a
    \]
    to be the linear function whose graph over $[0,1]$ concides with the segment $\varphi_{R}^{-1}(\ell)$.

    Pick a large integer $p$ so that $\frac{1}{m^p}<\frac{m^ku}{n^k}$ and $\frac{1}{n^{p+2}}<\frac{1}{m^{p+3}}$. Since $\frac{m^ku}{n^k}>\frac{1}{m^p}\geq \frac{2}{m^{p+1}}$, the interval $[a,a+\frac{m^ku}{n^k}]$ contains an $m^{p+1}$-adic interval, say $[\frac{j}{m^{p+1}},\frac{j+1}{m^{p+1}}]$. So
    \begin{equation}\label{eq:g0andg1}
        g(0) = a \leq \frac{j}{m^{p+1}} < \frac{j+1}{m^{p+1}} \leq a+\frac{m^ku}{n^k} = g(1).
    \end{equation}
    By our choices of $k$ and $p$,
    \begin{equation}\label{eq:noobliqueline2}
        0 < g\Big( \frac{t+1}{n^{p+2}} \Big)-g\Big( \frac{t}{n^{p+2}} \Big)  = \frac{m^ku}{n^k}\cdot\frac{1}{n^{p+2}}<\frac{1}{2}\cdot\frac{1}{m^{p+3}}, \quad 0\leq t\leq n^{p+2}-1.
    \end{equation}
    Combining this with~\eqref{eq:g0andg1}, we can find an index $0\leq t_0\leq n^{p+2}-1$ such that 
    \begin{equation}\label{eq:noobliqueline4}
        \frac{j}{m^{p+1}}+\frac{j_0}{m^{p+3}} \leq g\Big( \frac{t_0}{n^{p+2}} \Big) < g\Big( \frac{t_0+1}{n^{p+2}} \Big) \leq \frac{j}{m^{p+1}}+\frac{j_0+1}{m^{p+3}}
    \end{equation}
    because the difference between the rightmost and leftmost values is $m^{-(p+3)}$. 
    
    Consider the rectangle
    \[
        \widetilde{R} := \Big[ \frac{t_0}{n^{p+2}},\frac{t_0+1}{n^{p+2}} \Big] \times \Big[ \frac{j}{m^{p+1}},\frac{j}{m^{p+1}}+\frac{1}{m^{p+2}} \Big].
    \]
    By the definition of $\varphi_R$ and $Q$,
    \[
        \varphi_{\widetilde{R}}(Q) = \Big( \frac{t_0}{n^{p+2}}+\frac{i_0}{n^{p+3}},\frac{t_0}{n^{p+2}}+\frac{i_0+1}{n^{p+3}} \Big) \times \Big( \frac{j}{m^{p+1}}+\frac{j_0}{m^{p+3}}, \frac{j}{m^{p+1}}+\frac{j_0+1}{m^{p+3}} \Big).
    \] 
    Inequality~\eqref{eq:noobliqueline4} shows that the graph of $g$ over the interval $[\frac{t_0}{n^{p+2}}+\frac{i_0}{n^{p+3}},\frac{t_0}{n^{p+2}}+\frac{i_0+1}{n^{p+3}}]$ lies entirely inside the open rectangle $\varphi_{\widetilde{R}}(Q)$. Since this graph is a subset of $\varphi_{R}^{-1}(\ell)\subset K$, we have $\varphi_{\tilde{R}}(Q)\cap K\neq\varnothing$. 
    Since $Q$ is open,
    \[
        \varphi_{\widetilde{R}}(Q)\cap K = \varphi_{\widetilde{R}}(Q) \cap \varphi_{\widetilde{R}}(K) = \varphi_{\widetilde{R}}(Q\cap K).
    \]
    Thus $Q\cap K\neq\varnothing$, which is a contradiction.
\end{proof}

\subsection{Bedford-McMullen carpets: horizontal slices}
\label{subsec:bmcarpet2}

We turn to horizontal slices of Bedford-McMullen carpets. As will become evident later, these slices contain crucial information that plays a significant role in investigating the embedding problem. For a Bedford-McMullen carpet $K=K(n,m,\Lambda)$, it is convenient to adopt the following notation.

\begin{itemize}
    \item $J:=\{0\leq j\leq m-1: \exists i \text{ s.t. } (i,j)\in\Lambda\}$, which collects the digits of non-empty rows in the initial pattern. 
    \item $I:=\{0\leq i\leq n-1: \exists j \text{ s.t. } (i,j)\in\Lambda\}$, which collects the digits of non-empty columns in the initial pattern.
    \item For $0\leq j\leq m-1$, $I_j:=\{0\leq i\leq n-1: (i,j)\in\Lambda\}$, which collects the digits of selected rectangles in the $j$-th row.
    \item For $0\leq i\leq n-1$, $J_i:=\{0\leq j\leq m-1: (i,j)\in\Lambda\}$, which collects the digits of selected rectangles in the $i$-th column.
    \item $N:=\max_{0\leq j\leq m-1}\#I_j$.
    \item $K^y:=\{x: (x,y)\in K\}$ (horizontal slice).
\end{itemize}

Note that every point $y\in\pi_2(K)$ has an $m$-adic expansion given by $y=\sum_{k=1}^\infty y_km^{-k}$ with $y_k\in J$ for all $k\geq 1$, which is unique except possibly when $y$ belongs to the set of $m$-adic rationals $\bigcup_{k=1}^\infty\frac{\mathbb{Z}}{m^k}\cap(0,1)$. Furthermore, when the expansion is unique, we have
\begin{align}
    K^y &= \Big\{ x: \Big( x,\sum_{k=1}^\infty\frac{y_k}{m^k} \Big) \in K \Big\} \notag \\
    &= \Big\{ x: \exists \{x_k\}_{k} \text{ such that } x=\sum_{k=1}^\infty \frac{x_k}{n^k} \text{ and } (x_k,y_k)\in\Lambda \text{ for all } k  \Big\} \notag\\
    &= \Big\{ \sum_{k=1}^\infty \frac{x_k}{n^k}: x_k\in I_{y_k} \text{ for all } k \Big\}. \label{eq:expressionofky}
\end{align} 
For future use, we also define for $p\geq 1$ that
\begin{equation}\label{eq:kyp}
    K^y_p:= \bigcup\Big\{ [x,x+n^{-p}]: x\in \Big\{ \sum_{k=1}^p \frac{x_k}{n^k}: x_k\in I_{y_k} \text{ for } 1\leq k\leq p \Big\} \Big\}
\end{equation}
and call the interior of every interval in the above union an \emph{$n^p$-adic partition interval} of $K^y_p$. We remark that there is a subtle distinction between  $K^y_p$ and the union of all closed $n^p$-adic intervals that intersect $K^y$, since the intersection may occur only at an endpoint, which may make the union slightly larger than $K^y_p$. It is easy to see that $\{K^y_p\}_{p=1}^\infty$ is decreasing, $\bigcap_{p=1}^\infty K^y_p=K^y$ and 
\begin{equation}\label{eq:dimlbofky}
    \dimh K^y \leq \underline{\dim}_{\textup{B}}\, K^y \leq \liminf_{M\to\infty} \frac{\log (\prod_{k=1}^M \#I_{y_k})}{-\log n^M}\leq\frac{\log N}{\log n},
\end{equation}
where $\underline{\dim}_{\textup{B}}$ denotes the lower box dimension. Furthermore, a standard argument shows that $\h^{\log N/\log n}(K^y)\leq 1<\infty$ for all $y$; we omit the details. Here and afterwards, $\h^s$ denotes the $s$-dimensional Hausdorff measure. For a comprehensive treatment of dimensions and measures for sets defined by digit restrictions and frequency, see~\cite[Chapter 1]{BP17}.

The next lemma records several basic facts about horizontal slices.

\begin{lemma}\label{lem:bunchoffacts}
    Let $K=K(n,m,\Lambda)$ be a Bedford-McMullen carpet with $n>m$. Then for any $y\in\pi_2(K)$ with a unique $m$-adic expansion $y=\sum_{k=1}^\infty y_km^{-k}$ (where $y_k\in J$):
    \begin{enumerate}
        \item For every $p\geq 1$ and every pair of $n^p$-adic partition intervals $U,V$ of $K^y_p$, $K^y\cap U$ is simply a translated copy of $K^y\cap V$;
        \item For every $p,t\geq 1$ and every $n^p$-adic partition interval $U$ of $K^y_p$, $U\cap K^y_p$ contains exactly $\prod_{k=p+1}^{p+t}\# I_{y_k}$ many $n^{p+t}$-adic partition intervals of $K^y_{p+t}$. 
    \end{enumerate}
    In addition, if $\#I_{y_k}\geq 2$ for all large $k$, then the following hold:
    \begin{enumerate}
        \setcounter{enumi}{2} 
        \item $\inf_{p\geq 1}\min\{|K^y\cap U|/|U|: U\text{ is an  $n^p$-adic partition interval of $K^y_p$}\}>0$;
        \item Let an open interval $G$ be called a \emph{gap} of $K^y_p$ if $G$ is a connected component of $U\setminus K^y$ for some $n^p$-adic partition interval $U$ of $K^y_p$. Then 
        \[
            \sup_{p\geq 1}\max\{|G|n^{p}: G \text{ is a gap of } K^y_p\}<1.
        \]
    \end{enumerate}
\end{lemma}
\begin{proof}
    Let $p\geq 1$ and let $U$ be any $n^p$-adic partition interval of $K^y_p$, say $U=(a,a+n^{-p})$. By~\eqref{eq:expressionofky} and~\eqref{eq:kyp}, 
    \begin{align*}
        K^y \cap U &= K^y\cap (a,a+n^{-p}) \\
        &= \Big\{ a+\sum_{k=p+1}^\infty \frac{z_k}{n^k}: z_k\in I_{y_k} \text{ for all } k \Big\}\setminus\{a,a+n^{-p}\} \\
        &= a + \Big( \Big\{ \sum_{k=p+1}^\infty \frac{z_k}{n^k}: z_k\in I_{y_k} \text{ for all } k \Big\}\setminus\{0,n^{-p}\} \Big).
    \end{align*}
    This shows that $K^y\cap U$ is a translate of the same set for any such $U$ and thus proves (1). Also, 
    \[
        K^y_{p+t}\cap U = \bigcup\Big\{\Big[ a+\sum_{k=p+1}^{p+t} \frac{x_k}{n^k}, \Big( a+\sum_{k=p+1}^{p+t} \frac{x_k}{n^k} \Big)+n^{-p-t}\Big] : x_k\in I_{y_k} \Big\}\setminus\{a,a+n^{-p}\},
    \]
    which implies (2).
    
    Next, suppose $\#I_{y_k}\geq 2$ for all large $k$. Let $p$ be so large that $\#I_{y_k}\geq 2$ whenever $k\geq p$. Let $G\subset U$ be any gap of $K^y_p$. By (2), $U$ contains at least four $n^{p+2}$-adic partition intervals of $K^{y}_{p+2}$. Since these open intervals are disjoint and each of them contains some point in $K^y$, we have $|K^y\cap U|\geq n^{-p-2}$ (which proves (3)) and $|G|\leq |U|-2n^{-p-2}= (1-2n^{-2})n^{-p}$, which gives (4).
\end{proof}

The following lemma slightly simplifies the proof of Corollary~\ref{cor:leftrightendpt}.
\begin{lemma}\label{lem:fromqto1}
    Let $K=K(n,m,\Lambda)$ be a Bedford-McMullen carpet with $n>m$ and let $y\in\pi_2(K)$. If $h\in\S(\R)$ satisfies that $h(K^y)\subset K^y$ and there are two open intervals $U,V$ such that 
    \begin{equation}\label{eq:fromqto1first}
        h^q(K^y\cap U) = K^y \cap V \neq \varnothing
    \end{equation}
    for some $q>1$, then we can find open intervals $U',V'$ such that 
    \[
        h(K^y\cap U') = K^y \cap V' \neq \varnothing.
    \]
\end{lemma}
\begin{proof}
    Since $K^y\cap V=h^q(K^y\cap U)$ is non-empty and contained in $h^q(U)\cap V$, $h^q(U)\cap V$ is a non-empty open interval that meets $K^y$. Taking $V':=h^q(U)\cap V$ and $U':=h^{-1}(V')$, we have 
    \begin{align*}
        h(K^y\cap U') \subset h(K^y) \cap h(U') \subset K^y \cap V' \neq\varnothing.
    \end{align*}
    For the reverse inclusion, just note that
    \begin{align*}
        K^y \cap V' \subset K^y \cap V &= h^q(K^y \cap U) \cap V && (\text{by~\eqref{eq:fromqto1first}}) \\
        &\subset h^q(K^y) \cap h^q(U) \cap V \\
        &= h^q(K^y) \cap V'  && \text{(definition of $V'$)} \\
        &= h\big( h^{q-1}(K^y)\cap h^{-1}(V') \big) \\
        &\subset h(K^y \cap U'). && \text{(since $q\geq 1$)}
    \end{align*}
\end{proof}

It is also worthy of mentioning that the Hausdorff dimension of Bedford-McMullen carpets was obtained by Bedford~\cite{Bed84} and McMullen~\cite{Mcm84} independently as follows:
\begin{equation}\label{eq:dimhk}
    \dimh K = \frac{1}{\log m}\cdot\log\Big( \sum_{j=0}^{m-1} (\# I_j)^{\log m/\log n} \Big).
\end{equation}

\section{Logarithmic commensurability on non-stationary deleted-digit sets}
\label{sec:logcom}

Deleted-digit sets naturally appear in the study of horizontal and vertical slices and projections of Bedford-McMullen carpets.

\begin{definition}[Deleted-digit set]
	Let $n\geq 2$ be an integer. For any non-empty digit set $\D\subset\{0,1,\ldots,n-1\}$, we define $E(n,\D)$ as the self-similar set associated with the IFS $\{\frac{x+i}{n}:i\in\D\}$ on $\R$. Specifically, 
	\begin{equation*}
		E(n,\D) = \Big\{ \sum_{k=1}^\infty \varepsilon_kn^{-k}: \varepsilon_k\in\D \Big\}.
	\end{equation*}
\end{definition}

A well-known example is the middle-third Cantor set, where $n=3$ and $\D=\{0,2\}$. Note that for any deleted-digit set $E(n,\D)$, the associated IFS $\{\frac{x+i}{n}:i\in\D\}$ always satisfies the open set condition; see~\cite{Fal14}. Since $\frac{[0,1]+i}{n}\subset[0,1]$ for all $n\geq 1$ and all $0\leq i\leq n-1$, the unit interval $[0,1]$ can be used for the standard iteration process to construct $E(n,\D)$. More precisely, starting with $E_0(n,\D):=[0,1]$ and recursively defining 
\begin{equation}\label{eq:eknd}
    E_k(n,\D) := \bigcup_{i\in\D} \frac{E_{k-1}(n,\D)+i}{n}, \quad k\geq 1,
\end{equation}
we get a decreasing sequence $\{E_{k}(n,\D)\}_{k=0}^\infty$ such that $\bigcap_{k=0}^\infty E_k(n,\D)=E(n,\D)$.

It is evident that $E(n,\D)$ has non-empty interior if and only if $\#\D=n$. In particular, if $\dimh E(n,\D)<1$ then $E(n,\D)$ must be totally disconnected. It can also be easily checked by induction that $E_k(n,\D)$ is a union of $n^k$-adic intervals, and each pair of such intervals is either adjacent (sharing a common endpoint) or disjoint. 

Many horizontal and vertical slices of a Bedford-McMullen carpet $K=K(n,m,\Lambda)$ are deleted-digit sets or a finite union of them. For example, if $y=\sum_{k} jm^{-k}$ for some $j\in J$, then we have by~\eqref{eq:expressionofky} that $K^y = E(n,I_j)$. On the other hand, the projections of $K$ onto the principal axes are also deleted-digit sets: 
\begin{align*}
    \pi_1(K) &= \Big\{ \sum_{k=1}^\infty \frac{x_k}{n^k}: \Big( \sum_{k=1}^\infty \frac{x_k}{n^k},\sum_{k=1}^\infty \frac{y_k}{m^k} \Big) \in K \Big\} \\
    &= \Big\{ \sum_{k=1}^\infty \frac{x_k}{n^k}: \exists \{y_k\} \text{ such that } (x_k,y_k)\in\Lambda \Big\}\\
    &= \Big\{ \sum_{k=1}^\infty \frac{x_k}{n^k}: x_k\in I \text{ for all $k$} \Big\} = E(n,I)
\end{align*}
and similarly, $\pi_2(K)=E(m,J)$.

\begin{lemma}\label{lem:fw09}
    Let $E=E(n,\D)$ be a deleted-digit set with $1<\#\D<n$. If $g\in\S(\R)$ sends $E$ into itself, then the similarity ratio of $g$ is a rational power of $n$.
\end{lemma}
\begin{proof}
    The statement is a special case of~\cite[Theorem 1.1]{FW09}.
\end{proof}

A generalization of this logarithmic commensurability theorem as follows is one of the main ingredients in our proof of the dependent case. Recall that $N:=\max_{0\leq j\leq m-1}\#I_j$.

\begin{proposition}\label{prop:allnthenrational}
    Let $K=K(n,m,\Lambda)$ be a Bedford-McMullen carpet with $n>m$, $2\leq N\leq n-1$ and write $\alpha:=\frac{\log N}{\log n}$. For $y\in\pi_2(K)$, if $0<\h^\alpha(K^y)<\infty$ and there is a contraction $h\in\S(\R)$ sending $K^y$ into itself, then the contraction ratio of $h$ is a rational power of $n$.
\end{proposition}

We remark that~\cite[Theorem 1.1]{FW09} addresses self-similar sets that are built by repeating the same pattern at each level of the iterative process, whereas a horizontal slice is usually a deleted-digit set constructed from patterns that may vary from one level to another. Roughly speaking, Proposition~\ref{prop:allnthenrational} establishes that even in this non-stationary setting, if the patterns have consistent size---that is, the number of selected subintervals eventually stabilizes at every level---then the slice cannot be arbitrarily rescaled to fit within itself.

Our proof works by contradiction. If the shrinking factor is not a rational power of $n$, one could position the image of some local portion of the slice such that it either exceeds the allowable mass of the target region (see Case 1 below) or spans across a gap, resulting in an impossible hole in the initial structure (see Case 2 below). Both situations contradict the inherent geometric properties of the slice. Some idea of the proof stems from an ongoing work of the author and Rao~\cite{RXtodedone}.

\begin{proof}
    Considering $h^2$ if necessary, we may assume that $h$ is orientation-preserving, that is, $h(x)=\rho x+a$ for some $0<\rho<1$ and $a\in\R$. We aim to show that $\frac{\log\rho}{\log n}\in\mathbb{Q}$.
    
    If $y\in\bigcup_{k=1}^\infty\frac{\Z}{m^k}\cap(0,1)$, then it is not hard to see that $K^y$ is a finite union of scaled copies of $K^0$ and $K^1$, say 
    \[
        K^y = \Big( \bigcup_{s\in\Omega}\frac{K^0+s}{n^k} \Big) \cup \Big( \bigcup_{t\in\Gamma}\frac{K^1+t}{n^k} \Big),
    \]
    where $k\geq 1$ is an integer and $\Omega,\Gamma\subset\Z$. Since $h(K^y)\subset K^y$, Baire's theorem implies that one of those scaled copies, say $\frac{K^0+s_*}{n^k}$, must contain an (relative) interior part of $h(K^y)$ and hence of $\frac{K^0+s_*}{n^k}$ itself. From this observation, we can find a scaled copy $E$ of $K^0$ contracting by a factor $n^{-q}$ for some large integer $q$ such that $h(E)\subset \frac{K^0+s_*}{n^k}$. Since $K^0=E(n,I_0)$ is a deleted-digit set with positive Hausdorff dimension (because $\frac{K^0+s_*}{n^k}$ contains a copy of $h(K^y)$, which has dimension $\alpha>0$), we have by Lemma~\ref{lem:fw09} that $\frac{\log\rho}{\log n}\in\mathbb{Q}$.
    
    If $y\notin\bigcup_{k=1}^\infty\frac{\Z}{m^k}\cap(0,1)$ then, as seen before, $y$ admits a unique $m$-adic expansion $y=\sum_{k=1}^\infty y_km^{-k}$ with $\{y_k\}\subset J$. 
    We first prove that $\#I_{y_k}=N$ for all large $k$. Otherwise, $\#I_{y_k}<N$ for infinitely many $k$. Then for any $p\geq 1$, we have for all sufficiently large $\widetilde{M}$ that 
    \begin{equation}\label{eq:newproofofalln}
        \prod_{k=1}^{\widetilde{M}} \#I_{y_k} \leq (N-1)^pN^{\widetilde{M}-p} = \Big( \frac{N-1}{N} \Big)^p N^{\widetilde{M}}.
    \end{equation}
    Since $0<\h^\alpha(K^y)<\infty$, there is a constant $\delta_0$ such that $\h^\alpha_\delta(K^y)>\h^\alpha(K^y)/2>0$ when $0<\delta\leq\delta_0$ (for the definition of the measure $\h^\alpha_\delta$, see~\cite{Fal14}). However, by~\eqref{eq:newproofofalln} we obtain
    \[
        \h^\alpha_{n^{-\widetilde{M}}} (K^y) \leq \Big( \frac{N-1}{N} \Big)^p N^{\widetilde{M}} \cdot (n^{-\widetilde{M}})^\alpha = \Big( \frac{N-1}{N} \Big)^p < \frac{1}{2}\h^\alpha(K^y)
    \]
    provided that $p$ and $\widetilde{M}$ are large enough. This contradicts the lower bound above.

    Pick $p_0$ such that $\#I_{y_k}=N$ for all $k\geq p_0$. Recall the notation~\eqref{eq:kyp}. Since $N\leq n-1$, it is not hard to check that 
    \[
        M:= \sup_{p\geq p_0}\max\Big\{|B|n^p: B \text{ is a connected component of } K^y_p \Big\} \leq 2N<\infty,
    \]
    that is, for all $p\geq p_0$, every connected component of $K^y_p$ contains at most $2N$ many $n^p$-adic partition intervals of $K^y_p$. 
    Picking $p_0$ large enough at the beginning, we may assume that there is a connected component $B$ of $K^y_{p_0}$ consisting of exactly $M$ many  $n^{p_0}$-adic partition intervals of $K^y_{p_0}$. Note that $B$ is a closed interval. 
    
    Now, suppose on the contrary that $\frac{\log\rho}{\log n}\notin\Q$. Let $\varepsilon>0$ be a small fixed constant (will be specified later) and pick positive integers $s,t$ such that $-\varepsilon\leq s\cdot\frac{\log\rho}{-\log n}-t<0$. Equivalently, $n^{-t}<\rho^s\leq n^{\varepsilon-t}$. Since $h(K^y)\subset K^y$, $h^s(K^y)\subset K^y$. We distinguish two cases.

    {\bf Case 1}: $h^s(B)$ intersects exactly one connected component of $K^y_{p_0+t}$, say $B'$. In this case, $h^s(B\cap K^y)\subset B'\cap K^y$. Since $B$ consists of $M$ many $n^{p_0}$-adic partition intervals of $K^y_{p_0}$ and their intersections with $K^y$ are distinct from one another by only a translation (recall Lemma~\ref{lem:bunchoffacts}(1)), one can find a $\rho^sn^{-p_0}$-separated subset of $h^s(B\cap K^y)$ that has cardinality $M$. Since $\rho^sn^{-p_0}>n^{-p_0-t}$, every $n^{p_0+t}$-adic partition interval of $K^y_{p_0+t}$ in $B'$, which has length $n^{-p_0-t}$, contains at most one point in that subset. Thus by the maximality of $M$, $B'$ contains exactly $M$ many $n^{p_0+t}$-adic partition intervals of $K^y_{p_0+t}$. Recall that $0<\h^\alpha(K^y)<\infty$. Writing $c$ to be the $\alpha$-dimensional Hausdorff measure of the intersection of $K^y$ and any $n^{p_0+t}$-adic partition interval in $K^y_{p_0+t}$, we see from Lemma~\ref{lem:bunchoffacts}(1) that $0<c<\infty$ and $\h^\alpha(B'\cap K^y)=Mc$. On the other hand, by Lemma~\ref{lem:bunchoffacts}(2), each $n^{p_0}$-adic partition interval in $B$ contains exactly $\prod_{k=p_0+1}^{p_0+t} \#I_{y_k} = N^{t}$ many $n^{p_0+t}$-adic partition intervals in $K^y_{p_0+t}$. So $\h^\alpha(B\cap K^y)=MN^tc$ and 
    \begin{align}
        \h^\alpha(h^s(B\cap K^y)) = \rho^{\alpha s}\h^\alpha(B\cap K^y) &= \rho^{\alpha s}\cdot MN^tc \notag \\ 
        &> n^{-\alpha t}MN^tc \notag \\
        &= N^{-t}\cdot M\cdot N^tc = Mc = \h^\alpha(B'\cap K^y), \label{eq:rationpropcase1}
    \end{align}
    which contradicts that $h^s(B\cap K^y)\subset B'\cap K^y$.

    {\bf Case 2}: $h^s(B)$ intersects at least two connected components of $K^y_{p_0+t}$. In this case, $h^s(B)$ contains a hole between those components. More precisely, there exists some integer $i$ such that $(\frac{i}{n^{p_0+t}},\frac{i+1}{n^{p_0+t}})\cap K^y=\varnothing$ and $(\frac{i}{n^{p_0+t}},\frac{i+1}{n^{p_0+t}})\subset h^s(B)$. Write $V:=(\frac{i}{n^{p_0+t}},\frac{i+1}{n^{p_0+t}})$. Comparing the length, we can find either exactly one or two adjacent $n^{p_0}$-adic partition intervals of $K^y_{p_0}$ in $B$ of which the images under $h^s$ meet $V$.
    
    {\bf Case 2.1}: $V\subset h^s(\widetilde{I})$ for only one $n^{p_0}$-adic partition interval $\widetilde{I}$ in $B$ (see Figure~\ref{fig:vi1i2}(A)). Then, since $h^s(\widetilde{I} \cap K^y)\subset K^y$, $h^s(\widetilde{I}\cap K^y)\cap V=\varnothing$. This in turn tells us that $\widetilde{I}\cap K^y$ contains a gap of length $\geq|h^{-s}(V)|$. Therefore, 
    \begin{equation}\label{eq:rationpropcase2}
        \max\{|G|n^{p_0}: G \text{ is a gap of } K^y_{p_0}\} \geq |h^{-s}(V)|n^{p_0} = \rho^{-s}n^{-p_0-t}\cdot n^{-p_0} > n^{-\varepsilon}.
    \end{equation}
    Choosing $\varepsilon$ small enough at the beginning, this contradicts Lemma~\ref{lem:bunchoffacts}(4). 
    
    {\bf Case 2.2}: There are two $n^{p_0}$-adic partition intervals $\widetilde{I}_1,\widetilde{I}_2$ of $K^y_{p_0}$ in $B$ such that $V\cap h^s(\widetilde{I}_i)\neq\varnothing$, $1\leq i\leq 2$. Without loss of generality, assume that $\widetilde{I}_1$ is to the left of $\widetilde{I}_2$. Denote by $a_i$ the left endpoint of $h^s(\widetilde{I}_i)$. See Figure~\ref{fig:vi1i2}(B) for an illustration. 
    
    \begin{figure}[htbp]
        \centering
        \subfloat[One-interval case]
        {
            \begin{minipage}[t]{155pt}
                \centering
                \begin{tikzpicture}[scale=1]
                    \draw[thick] (0.1,0) to (1.9,0);
                    \node at(1,0.4) {$h^s(\widetilde{I})$};
                    \node[draw,shape=circle,inner sep=1.5pt,thick] at(0,0) {};
                    \node[draw,shape=circle,inner sep=1.5pt,thick] at(2,0) {};
                    \node[draw,shape=circle,inner sep=1.5pt,thick] at(0.2,-1) {};
                    \node[draw,shape=circle,inner sep=1.5pt,thick] at(1.8,-1) {};
                    \draw[thick] (0.3,-1) to (1.7,-1);
                    \node at(1,-1.3) {$V$};
                    \draw[thick,dashed] (0.2,-1) to (0.2,0);
                    \draw[thick,dashed] (1.8,-1) to (1.8,0);
                \end{tikzpicture}
            \end{minipage}
        }
        \subfloat[Two-intervals case]
        {
            \begin{minipage}[t]{155pt}
                \centering
                \begin{tikzpicture}[scale=1]
                    \draw[thick] (0.1,0) to (1.9,0);
                    \draw[thick] (2.1,0) to (3.9,0);
                    \node at(1,0.4) {$h^s(\widetilde{I}_1)$};
                    \node at(3,0.4) {$h^s(\widetilde{I}_2)$};
                    \node[draw,shape=circle,inner sep=1.5pt,thick] at(0,0) {};
                    \node[draw,shape=circle,inner sep=1.5pt,thick] at(2,0) {};
                    \node[draw,shape=circle,inner sep=1.5pt,thick] at(4,0) {};
                    \node at(0,-0.3) {$a_1$};
                    \node at(2,-0.3) {$a_2$};
                    \node[draw,shape=circle,inner sep=1.5pt,thick] at(1,-1) {};
                    \node[draw,shape=circle,inner sep=1.5pt,thick] at(2.6,-1) {};
                    \draw[thick] (1.1,-1) to (2.5,-1);
                    \node at(1.8,-1.3) {$V$};
                    \draw[thick,dashed] (1,-1) to (1,0);
                    \draw[thick,dashed] (2,-1) to (2,0);
                    \draw[thick,dashed] (2.6,-1) to (2.6,0);
                \end{tikzpicture}
            \end{minipage}
        }
        \caption{A local illustration of Case 2}
        \label{fig:vi1i2}
    \end{figure}
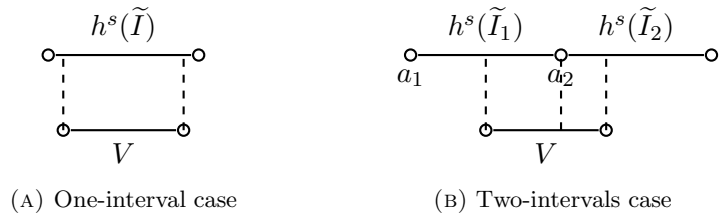
    
    By Lemma~\ref{lem:bunchoffacts}(1), 
    \begin{align*}
        \big( a_1,a_1+|h^s(\widetilde{I}_2)\cap V| \big)\cap K^y &= \big( a_2,a_2+|h^s(\widetilde{I}_2)\cap V| \big)\cap K^y \\
        &= h^s(\widetilde{I}_2)\cap V \cap K^y \\
        &= \varnothing.
    \end{align*}
    Thus 
    \[
        h^s(\widetilde{I}_1)\cap K^y \subset (h^s(\widetilde{I}_1)\setminus V) \setminus \big( a_1,a_1+|h^s(\widetilde{I}_2)\cap V| \big).
    \]
    But the right hand side is an interval of length 
    \begin{equation}\label{eq:rationpropcase3}
        |h^s(\widetilde{I}_1)| -|h^s(\widetilde{I_1})\cap V|- |h^s(\widetilde{I}_2)\cap V| = \rho^s n^{-p_0}-|V| < n^{\varepsilon-p_0-t}-n^{-p_0-t}.
    \end{equation} 
    Choosing $\varepsilon$ small enough at the beginning, this contradicts Lemma~\ref{lem:bunchoffacts}(3). 
\end{proof}

The above argument actually reveals more than the logarithmic commensurability. In fact, since the contraction ratio of $h$ is a rational power of $n$, iterating the embedding will yield an exact match between certain regions of $K^y$, which forces boundary points to map to corresponding boundary points. To make this precise, for any compact set $A\subset\R$ and $a\in A$, we call $a$ a \emph{left} (resp. \emph{right}) \emph{isolated point} of $A$ if there is some $r>0$ such that $(a-r,a)\cap A=\varnothing$ (resp. $(a,a+r)\cap A=\varnothing$).  

\begin{corollary}\label{cor:leftrightendpt}
    Let $K,N,\alpha,y,h$ be as in Proposition~\ref{prop:allnthenrational}. Then there exist left and right isolated points $a$ and $b$ of $K^y$ such that $h(a)$ and $h(b)$ are also left and right isolated points of $K^y$, respectively (with the order reversed if $h$ is orientation-reversing). Moreover, for all sufficiently large $k$, one can find level-$k$ rectangles $R_k$ and $R'_k$ satisfying $(a,y)\in\varphi_{R_k}(K)$, $(b,y)\in\varphi_{R'_k}(K)$, and $\pi_2(R_k)=\pi_2(R'_k)$.
\end{corollary}

Roughly speaking, the ``moreover'' part requires that the two boundary points $(a,y),(b,y)$ lie in the same row at every stage of the iterative construction. This condition is needed for technical reasons that will become clear in later proofs. 

\begin{proof}
    To address the first statement, it suffices to prove it when $h$ is orientation-preserving. In fact, if $h$ is orientation-reversing, then $h^2$ is orientation-preserving. If $a,h^2(a)$ are both left (resp. right) isolated points of $K^y$, then $h(a)$ must be a right (resp. left) isolated point and thus completes the proof. 
    
    Similarly as before, let $\rho$ be the contraction ratio of $h$. By Proposition~\ref{prop:allnthenrational}, $\rho=n^{-p/q}$ for some positive integers $p,q$. We distinguish two cases.

    {\bf Case 1}: $y\notin\bigcup_{k=1}^\infty \frac{\mathbb{Z}}{m^k}\cap (0,1)$.

    Applying the argument in the proof of Proposition~\ref{prop:allnthenrational} to $s=q$, $t=p$ and $\varepsilon=s\cdot\frac{\log\rho}{\log n}-t=0$, we see that Case 2 in that proof is impossible because:
    \begin{enumerate}
        \item in Case 2.1, the right hand side of~\eqref{eq:rationpropcase2} now equals $1$, meaning that $\widetilde{I}$ contains a gap of the same length as itself, contradicting $\widetilde{I}\cap K^y\neq\varnothing$;
        \item in Case 2.2, the right hand side of~\eqref{eq:rationpropcase3} now equals $0$, meaning that $|h^q(\widetilde{I_1}\cap K^y)|=0$, contradicting $\h^\alpha(h^q(\widetilde{I_1}\cap K^y))>0$.
    \end{enumerate}
    Meanwhile, Case 1 in that proof now gives $\h^\alpha(h^q(B\cap K^y))=\h^\alpha(B'\cap K^y)$. We claim that 
    \begin{equation}\label{eq:hqkyhasnonemptyinte}
        h^q(B\cap K^y)=B'\cap K^y.
    \end{equation}
    If not, then since $h^q(B\cap K^y)\subset B'\cap K^y$, there exists some $x\in B'\cap K^y\setminus h^q(B\cap K^y)$. So we can find a large $k$ and an $n^k$-adic partition interval $W$ of $K^y_k$ such that $x\in W\cap B'$ but $W\cap h^q(B\cap K^y)=\varnothing$. However, since $0<\h^\alpha(K^y)<\infty$, Lemma~\ref{lem:bunchoffacts}(1) yields that $\h^\alpha(W\cap K^y)>0$. Thus 
    \[
        \h^\alpha(h^q(B\cap K^y)) \leq \h^\alpha(B'\cap K^y)-\h^\alpha(W\cap K^y) < \h^\alpha(B'\cap K^y),
    \]
    which contradicts the equality of the Hausdorff measures.
    
    Combining~\eqref{eq:hqkyhasnonemptyinte} and Lemma~\ref{lem:fromqto1}, we can find open intervals $U,V$ such that $h(K^y\cap U)=K^y\cap V\neq\varnothing$. Consequently, $h$ maps all the left and right isolated points of $K^y\cap U$ to the corresponding isolated points of $K^y\cap V$, respectively, which establishes the first statement.

    For the second one, since $y\notin\bigcup_{k=1}^\infty\frac{\Z}{m^k}\cap(0,1)$, the required level-$k$ rectangle $R_k$ and $R'_k$ are both unique for all $k$. Moreover, $y$ lies within the (relative) interior of both $\pi_2(R_k)$ and $\pi_2(R'_k)$, which implies that $\pi_2(R_k)=\pi_2(R'_k)$. 

    {\bf Case 2}: $y\in\bigcup_{k=1}^\infty \frac{\mathbb{Z}}{m^k}\cap(0,1)$.
    
    Similarly as at the beginning of the proof of Proposition~\ref{prop:allnthenrational}, we can express $K^y$ as 
    \begin{equation}\label{eq:whereisgamma}
        K^y = \Big( \bigcup_{i\in\Omega} \frac{K^0+i}{n^k} \Big) \cup \Big( \bigcup_{j\in\Gamma} \frac{K^1+j}{n^k} \Big),
    \end{equation}
    where $k$ is a large integer and $\Omega,\Gamma\subset\Z$. 
    Using Baire's theorem again, we may assume without loss of generality that some $\frac{K^0+i}{n^k}$ contains an interior part of $h(K^y)$. In particular, 
    \begin{equation}\label{eq:koplusiincludeh}
        \frac{K^0+i}{n^k} \supset h\Big( \frac{K^0+z}{n^c} \Big)
    \end{equation}
    for some $z,c\in\mathbb{Z}$. This inclusion can be reformulated as $g(K^0)\subset K^0$ for some similitude $g$ with contraction ratio $\rho n^{k-c}$. Since $0$ belongs to the last case, there are open intervals $U,V$ such that $g(U\cap K^0)=V\cap K^0\neq\varnothing$. In particular, $g$ maps all the one-sided isolated points of $K^0$ lying in $U$ to those lying in $V$. Since $g(K^0)\subset K^0$ is just a reformulation of~\eqref{eq:koplusiincludeh}, $h$ maps all one-sided isolated points of $\frac{K^0+z}{n^c}$ in some open interval $U'$ to one-sided isolated points of $\frac{K^0+i}{n^k}$ in another open $V'$. It remains to verify that these points are also one-sided isolated points of $K^y$ (for all large $k$, one can pick $R_k,R'_k$ to be those that lie above adjacent to $\R\times\{y\}$ and thus $\pi_2(R_k)=\pi_2(R'_k)$).
    
    To get the desired statement, we first note that $V'$ can be picked properly so that 
    \begin{equation}\label{eq:koslocallygoodv}
        \frac{K^0+i}{n^k}\cap V'=K^y\cap V'.
    \end{equation}
    This is automatic if $j\neq i$ for all $j\in\Gamma$ (with $\Gamma$ as in~\eqref{eq:whereisgamma}). The only potential complication occurs when some $j$ equals $i$, in which case $\frac{K^0+i}{n^k}$ and $\frac{K^1+j}{n^k}=\frac{K^1+i}{n^k}$ may overlap. If $K^0=K^1$ then it does not make any difference. If $K^0\neq K^1$, then because $\frac{K^0+i}{n^k}$ contains an interior part of $h(K^y)$, we have $\dimh K^0\geq \dimh K^y\geq \dimh K^1$. Since $K^0$ and $K^1$ are distinct deleted-digit sets, they must have different digit patterns; in particular, $K^0\setminus K^1\neq\varnothing$. So there is a small open interval $V''\subset V'$ such that 
    \[
        \Big( \frac{K^0+i}{n^k}\cap V'' \Big) \cap \frac{K^1+i}{n^k}=\varnothing \quad\text{but}\quad \frac{K^0+i}{n^k}\cap V''\neq\varnothing.
    \]
    Replacing $V'$ with $V''$ yields~\eqref{eq:koslocallygoodv}. In particular, every left (resp. right ) isolated point of $\frac{K^0+i}{n^k}$ in $V'$ is a left (resp. right) isolated point of $K^y$. 
    
    Next, suppose that a one-sided isolated point $a$ of $\frac{K^0+z}{n^k}$ in $U'$ is not a one-sided isolated point of $K^y$, there is a sequence $\{x_n\}\subset K^y\cap U'$ converging to $a$ from either the left or the right. Then $h(x_n)\to h(a)$ from the same side. But $h(x_n)\in K^y \cap V'$, which contradicts the fact that $h(a)$ is a one-sided isolated point of $K^y$ in $V'$ and completes the proof.
\end{proof}

\section{The independent case}\label{sec:indep}

One significant benefit of the independence assumption regarding $n$ and $m$ is that it yields a clear description of the dimensions of the orthogonal projections of $K$ onto all lines.

\begin{lemma}[\cite{FJS10}]\label{lem:projofcarpet}
    Let $K=K(n,m,\Lambda)$ be a Bedford-McMullen carpet. If $n\nsim m$, then $\dimh \pi_L(K)=\min\{\dimh K,1\}$ for every oblique linear line $L$.
\end{lemma}

\begin{theorem}\label{thm:notoblique}
    Let $K=K(n,m,\Lambda)$ be a Bedford-McMullen carpet that is not supported on any line, where $n>m\geq 2$, and let $f\in\S(\R^2)$ be a self-embedding similitude of $K$. If $n\nsim m$, then $f$ is not oblique.
\end{theorem}

As aforementioned, the result originates from the work of Algom and Hochman~\cite{AH19}. Below we present a new proof of it by considering three cases, only one of which involves the independence condition $n\nsim m$. 
Recall that $N:=\max_j\#I_j$.

\paragraph{\noindent{\bf Case 1}} $\#J<m$. In other words, there is an empty row in the initial pattern of $K$. 

\begin{proof}[Proof of Theorem~\ref{thm:notoblique} under Case 1]
    If $f$ is oblique, then by Lemma~\ref{lem:projofcarpet},
    \begin{equation}\label{eq:pitklowerbd}
        \dimh \pi_t(K) \geq \dimh \pi_t(f(K)) = \min\{\dimh K,1\}, \quad t=1,2.
    \end{equation}
    Since $\# J<m$, $\dimh \pi_2(K)<1$. If there is some $j$ such that $\#I_j\geq 2$ (recall the notation from Section~\ref{subsec:bmcarpet2}), then we also have $\dimh \pi_2(K)<\dimh K$ (recall~\eqref{eq:dimhk}) and thus 
    \[
        \dimh \pi_2(K)<\min\{1,\dimh K\},
    \]
    which contradicts~\eqref{eq:pitklowerbd}. 
    If $\#I_j=1$ for all $j\in J$, then $\#\Lambda=\#J$. Therefore, 
    \[
        \dimh\pi_1(K) = \dimh E(n,I) \leq \frac{\log\#\Lambda}{\log n} < \frac{\log\#J}{\log m} = \dimh\pi_2(K) \leq \min\{\dimh K,1\},
    \]
    which again contradicts~\eqref{eq:pitklowerbd}.
\end{proof}

\paragraph{\noindent{\bf Case 2}} $\#J=m$ and $N=n$. In other words, there is no empty row but a full row in the initial pattern of $K$.

\begin{proof}[Proof of Theorem~\ref{thm:notoblique} under Case 2]
    In this case, choose $j$ with $\#I_j=n$. Setting $y:=\sum_{k=1}^\infty jm^{-k}$, we have seen that $K^{y}=E(n,I_j)$, which is an interval. So if $f$ is oblique, then $K\supset f(K)$ contains an oblique segment, contradicting Proposition~\ref{prop:noobliquelines}.
\end{proof}

\paragraph{\noindent{\bf Case 3}} $\#J=m$ but $N<n$. In other words, one can find a selected rectangle and an unselected one in each row in the initial pattern of $K$. 

We sketch the rough idea as follows. Assume an oblique embedding exists; then the carpet contains an oblique thin copy of itself. Rescaling this copy repeatedly flattens it until the image becomes extremely horizontal. Such a flattened image contains a thin horizontal strip spanning a full unit width while fitting strictly into a row of the carpet's grid. Since every horizontal slice of the carpet is nonempty, the strip must intersect the carpet in a dense manner—yet it sits entirely within a row-gap where the carpet is empty, yielding the desired contradiction.

\begin{proof}[Proof of Theorem~\ref{thm:notoblique} under Case 3]
    In this case, $K$ does not contain any horizontal line segments, and $\pi_2(K)=E(m,J)=[0,1]$. The latter implies that every horizontal slice of $K$ is non-empty. Let $0<\lambda\leq 1$ be the similarity ratio of $f$. If $f$ is oblique, without loss of generality, assume that $f$ sends the $x$-axis to a line of slope $u>0$. Fix a large integer $p>0$ so that $m^{-p}\lambda\leq\frac{1}{9m^2n}$.

    Let us pick a sequence $\{R_k\}_{k=1}^\infty$, where each $R_k$ is a level-$k$ rectangle and $R_{k+1}\subset R_k$ for all $k\geq 1$. Denote by $a_k$ the left bottom vertex of $R_k$. Pick an arbitrary $n^{-k+p}\times m^{-k+p}$ grid rectangle $Q_k$---that is, $Q_k$ is the Cartesian product of a pair of $n^{k-p}$-adic and $m^{k-p}$-adic intervals---containing $f(a_k)$. For convenience, we also let $\ell^k,\ell_k$ be the bottom and left sides of $R_k$, respectively. So $\ell^k\cap\ell_k=\{a_k\}$. See Figure~\ref{fig:rkfrkinvfrk} for an illustration. Below we record two simple yet crucial facts. Here we slightly abuse notation by writing $\varphi_{Q_k}$ to be the natural affine map sending $[0,1]^2$ onto $Q_k$.

    \begin{enumerate}
        \item Since $\ell^k$ is parallel to the $x$-axis and $f(x\text{-axis})$ is supported on a line of slope $u>0$, the line supporting $\varphi_{Q_k}^{-1}f(\ell^k)$ has slope $\frac{m^{k-p}u}{n^{k-p}}$, which tends to $0$ as $k\to\infty$. Moreover, $\varphi_{Q_k}^{-1}f(\ell^k)$ has length 
        \begin{align}
            |\varphi_{Q_k}^{-1}f(\ell^k)| &= \sqrt{n^{2(k-p)}|\pi_1(f(\ell^k))|^2+m^{2(k-p)}|\pi_2(f(\ell^k))|^2} \notag \\ 
            &= \sqrt{n^{2(k-p)}\cdot\lambda^2|\ell^k|^2\cdot\frac{1}{1+u^2} + m^{2(k-p)}\cdot\lambda^2|\ell^k|^2\cdot\frac{u^2}{1+u^2}} \notag \\
            &= \sqrt{n^{2(k-p)}\cdot\frac{\lambda^2}{n^{2k}}\cdot\frac{1}{1+u^2} + m^{2(k-p)}\cdot\frac{\lambda^2}{n^{2k}}\cdot\frac{u^2}{1+u^2}} \notag \\
            &\to \frac{n^{-p}\lambda}{\sqrt{1+u^2}}, \quad k\to\infty \label{eq:lengthofvflk}
        \end{align}
        \item Similarly, the line supporting $\varphi_{Q_k}^{-1}f(\ell_k)$ has slope $-\frac{m^{k-p}}{n^{k-p}u}$, which also tends to $0$ as $k\to\infty$. Moreover, $\varphi_{Q_k}^{-1}f(\ell_k)$ has length
        \begin{align*}
            |\varphi_{Q_k}^{-1}f(\ell_k)| &= \sqrt{n^{2(k-p)}|\pi_1(f(\ell_k))|^2+m^{2(k-p)}|\pi_2(f(\ell_k))|^2} \\
            &= \sqrt{n^{2(k-p)}\cdot\frac{\lambda^2}{m^{2k}}\cdot\frac{1}{1+u^{-2}}+m^{2(k-p)}\cdot\frac{\lambda^2}{m^{2k}}\cdot\frac{u^{-2}}{1+u^{-2}}} \\
            &\to\infty, \quad k\to\infty.
        \end{align*}
    \end{enumerate}
    Roughly speaking, for all large $k$, the set $\varphi_{Q_k}^{-1}f(R_k)$ forms a very flat parallelogram with one side short and the other one much longer.

    \begin{figure}[htbp]
        \centering
        \begin{tikzpicture}[scale=0.6]
            \draw[thick] (0,0) to (1,0) to (1,5) to (0,5) to (0,0);
            \node at (0,0)[circle,fill,inner sep=1.5pt, red]{};
            \node at (-0.4,-0.4)[circle,inner sep=2pt,red]{$a_k$};
            \node at (0.5,2.5)[circle,inner sep=2pt]{$R_k$};
            \node at (2,2.5)[circle, inner sep=2pt]{$\xrightarrow{f}$};
            \draw[thick] (3-8/11,54/22+3/2) to (3+10/11,1.5) to (3+14/11,1.5+8/33) to (3-4/11,54/22+3/2+8/33) to (3-8/11,54/22+3/2);
            \node at (3+10/11,1.5)[circle,fill, inner sep=1.5pt, red]{};
            \node at (3+10/11+0.9,1.2)[circle, inner sep=2pt, red]{$f(a_k)$};
            \draw[thick,dashed] (3,0) to (4,0) to (4,5) to (3,5) to (3,0);
            \node at (3.5,0.5)[circle,inner sep=2pt]{$Q_k$};
            \node at (5.5,2.5)[circle, inner sep=2pt]{$\xrightarrow{\varphi_{Q_k}^{-1}}$};
            \draw[thick] (7,1.5+5/4) to (7+25/4,1.5) to (7+15/2,1.5+1/4) to (7+5/4,3) to (7,1.5+5/4);
            \node at (7+13/4,1.3)[circle,inner sep=2pt]{$\varphi_{Q_k}^{-1}f(R_k)$};
            \node at (7+25/4,1.5)[circle,fill, inner sep=1.5pt,red]{};
            \node at (7+25/4,0.9)[circle, inner sep=2pt,red]{$\widetilde{a}_k$};
            \node at (-0.4,2) {$\ell_k$};
            \node at (0.5,-0.4) {$\ell^k$};
        \end{tikzpicture}
        \caption{The evolution from $R_k$ to $f(R_k)$ to $\varphi_{Q_k}^{-1}f(R_k)$ (color online)}
        \label{fig:rkfrkinvfrk}
    \end{figure}
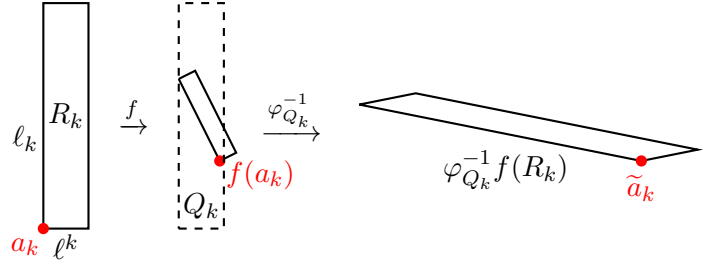
    
    Write $\widetilde{a}_k :=\varphi_{Q_k}^{-1}f(a_k)$. Since $f(a_k)\in Q_k$, $\widetilde{a}_k\in [0,1]^2$. Fix a large $k$ (will be specified later) so that $\frac{m^{k-p}}{n^{k-p}}\cdot\max\{u^{-1},u\}<\lambda$ and $|\varphi_{Q_k}^{-1}f(\ell^k)|\leq \frac{2n^{-p}\lambda}{\sqrt{1+u^2}}$ (recall~\eqref{eq:lengthofvflk}). Let us decompose the parallelogram $\varphi_{Q_k}^{-1}f(R_k)$ into vertical strips of unit width by setting
    \[
        \widetilde{E}_t := \varphi_{Q_k}^{-1}f(R_k) \cap \{z\in\R^2: \pi_1(z)\in[-t,-t+1]\}, \quad t\in \mathbb{Z}.
    \]
    Then $\varphi_{Q_k}^{-1}f(R_k)=\bigcup_{t\in\Z} \widetilde{E}_t$. 
    Note that for every $t$,
    \begin{align}
        |\pi_2(\widetilde{E}_t)| &\leq |\pi_2(\varphi_{Q_k}^{-1}f(R_k))| \notag \\
        &\leq |\pi_2(\varphi_{Q_k}^{-1}f(\ell^k))|+|\pi_2(\varphi_{Q_k}^{-1}f(\ell_k))| \notag \\
        &= m^{k-p}\cdot\frac{\lambda}{n^k}\cdot\frac{u}{\sqrt{1+u^2}}+m^{k-p}\cdot\frac{\lambda}{m^k}\cdot\frac{u^{-1}}{\sqrt{1+u^{-2}}}  < 2m^{-p}\lambda.\label{eq:lessthan13m}
    \end{align}

    Let us begin by proving Case 3 under the presumption of the following Condition A, and then proceed to establish its validity.

    \paragraph{{\bf Condition A}} There are integers $t_0\geq -1$ and $0\leq j_0\leq m$ such that $|\pi_1(\widetilde{E}_{t_0})|=1$ while $\pi_2(\widetilde{E}_{t_0})\subset [\tfrac{j_0}{m},\tfrac{j_0+1}{m}]$.

    Assume that Condition A holds. Let $j'_0:=j_0\pmod {m}$. Recall that $\varphi_{Q_k}^{-1}f(\ell_k)$ is supported on a line of slope $-\frac{m^{k-p}}{n^{k-p}u}$, say $g(x)=-\frac{m^{k-p}}{n^{k-p}u}x+c$ (the equation of that line). Since $\#I_{j'_0}<n$, there is some $0\leq i_0\leq n-1$ such that $K$ does not intersect the open rectangle $(\frac{i_0}{n},\frac{i_0+1}{n})\times(\frac{j'_0}{m},\frac{j'_0+1}{m})=:U$. Write $e$ to be the point $(-t_0,\lfloor\frac{j_0}{m}\rfloor)$ and $x_0:=-t_0+\frac{i_0+1/2}{n}$. Then the point $\xi_0:=(x_0,g(x_0))$ lies on the segment $\varphi_{Q_k}^{-1}f(\ell_k)$, and 
    \begin{align*}
        U+e &= \Big( \frac{i_0}{n}-t_0,\frac{i_0+1}{n}-t_0 \Big)\times \Big( \frac{j'_0}{m}+\lfloor\frac{j_0}{m}\rfloor, \frac{j'_0+1}{m}+\lfloor\frac{j_0}{m}\rfloor \Big) \\
        &= \Big( \frac{i_0}{n}-t_0,\frac{i_0+1}{n}-t_0 \Big)\times\Big( \frac{j_0}{m}, \frac{j_0+1}{m} \Big).
    \end{align*} 
    Since $\pi_2(\widetilde{E}_{t_0})\subset [\tfrac{j_0}{m},\tfrac{j_0+1}{m}]$, it is not hard to check that 
    \begin{equation}\label{eq:sliceinupluse}
        \xi_0+\varphi_{Q_k}^{-1}f(\ell^k) \subset U+e.
    \end{equation}
    See Figure~\ref{fig:localofthehole} for an illustration.

    \begin{figure}[htbp]
        \centering
        \begin{tikzpicture}[scale=1]
            \draw[thick,dashed] (0,0) to (0,4);
            \draw[thick,dashed] (4,0) to (4,4);
            \draw[thick,dashed] (-1.5,1) to (5.5,1);
            \draw[thick,dashed] (-1.5,2.5) to (5.5,2.5);
            \node at (2,3.5) {$\widetilde{E}_{t_0}$};
            \draw[thick] (-1, 1.75) to (5,1.45);
            \draw[thick] (-1,2.05) to (5.5,1.725);
            \draw[thick,red] (5,1.45) to (5.5,1.725);
            \draw[thick,red] (2,1.6) to (2.5,1.875);
            \node at (2,1.6)[circle,fill, inner sep=1.5pt]{};
            \node at (2,1.3) {$\xi_0$};
            \node at (5.8,1.3) [red] {$\varphi_{Q_k}^{-1}f(\ell^k)$};
            \node at (-1.9,1) {$\tfrac{j_0}{m}$};
            \node at (-1.9,2.5) {$\tfrac{j_0+1}{m}$};
            \draw[thick,blue] (1.4,1) to (2.6,1) to (2.6,2.5) to (1.4,2.5) to (1.4,1);
            \node at (2,0.7) [blue] {$U+e$};
        \end{tikzpicture}
        \caption{$\varphi_{Q_k}^{-1}f(\ell^k)+\xi_0$ is contained in the hole $U+e$ (color online)}
        \label{fig:localofthehole}
    \end{figure}
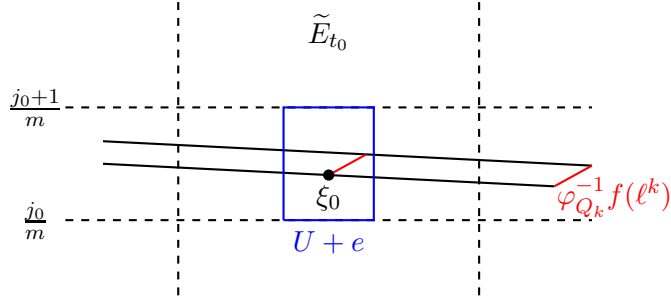

    Since the segment $\xi_0+\varphi_{Q_k}^{-1}f(\ell^k)$ is a translated copy of $\varphi_{Q_k}^{-1}f(\ell^k)$ and is contained in $\varphi_{Q_k}^{-1}f(R_k)$, $f^{-1}\varphi_{Q_k}(\xi_0)+\ell^k$ is a translated copy of $\ell^k$ and is contained in $R_k$. Since $\ell^k$ is the bottom side of the level-$k$ rectangle $R_k$, $\ell^k\cap \varphi_{R_k}(K)=\varphi_{R_k}(K^0\times\{0\})$. So there is $y_0\in[0,1]$ such that
    \[
        (f^{-1}\varphi_{Q_k}(\xi_0)+\ell^k)\cap\varphi_{R_k}(K) = \varphi_{R_k}(K^{y_0}\times\{y_0\})\neq\varnothing,
    \]
    where the non-empty conclusion is because $K^{y}\neq\varnothing$ for all $y$ (as pointed out at the beginning of Case 3). 
    This in turn implies that 
    \begin{equation}\label{eq:tranisanotherhorislice}
        (\xi_0+\varphi_{Q_k}^{-1}f(\ell^k)) \cap \varphi_{Q_k}^{-1}f\varphi_{R_k}(K) = \varphi_{Q_k}^{-1}f\varphi_{R_k}(K^{y_0}\times\{y_0\}) \neq\varnothing.
    \end{equation}
    Thus 
    \begin{align*}
        \varnothing &\neq (\xi_0+\varphi_{Q_k}^{-1}f(\ell^k)) \cap \varphi_{Q_k}^{-1}f\varphi_{R_k}(K) && \text{(just~\eqref{eq:tranisanotherhorislice})} \\
        &\subset (U+e) \cap \varphi_{Q_k}^{-1}f\varphi_{R_k}(K) && \text{(by~\eqref{eq:sliceinupluse})} \\
        &\subset (U+e) \cap \varphi_{Q_k}^{-1}f(K) \\
        &\subset (U+e) \cap \Big( \bigcup_{\substack{\text{$R$ is a level-$(k-p)$ rectangle} \\ \text{$R\cap f(K)\neq\varnothing$}}} \varphi_{Q_k}^{-1}\varphi_R(K) \Big) && \text{(since $f(K)\subset K$)} \\
        &\subset (U+e) \cap \bigcup_{z\in\Z^2}(K+z) \\
        &= (U+e) \cap (K+e) = \varnothing, && \text{(since $e$ is a lattice point)}
    \end{align*}
    which leads to a contradiction.

    It remains to prove Condition A. The key is to note that for every $y$, $\varphi_{Q_k}^{-1}f(R_k)\cap(\R\times\{y\})$ (which is a horizontal slice of the parallelogram) has length at most $Cn^{-p}\lambda$. To see this, consider Figure~\ref{fig:thetriangle} and recall that the two acute edges of that triangle has slope $\frac{m^{k-p}u}{n^{k-p}}$ and $-\frac{m^{k-p}}{n^{k-p}u}$, respectively. Note that $d_1\approx|\varphi_{Q_k}^{-1}f(\ell^k)|\lesssim n^{-p}\lambda$ and 
    \[
        d_2 = \frac{h}{\frac{m^{k-p}}{n^{k-p}u}} = u^2\frac{h}{\frac{m^{k-p}u}{n^{k-p}}} = u^2d_1 \lesssim n^{-p}\lambda.
    \]
    Thus $d_1+d_2$ is at most a constant multiple of $n^{-p}\lambda$.

    \begin{figure}[htbp]
        \centering
        \begin{tikzpicture}[scale=1]
            \draw[thick] (0,0) to (2,0.5);
            \draw[thick,dashed] (2,-0.25) to (-4,0.5);
            \draw[thick,dashed] (2,0.5) to (6,0); 
            \draw[thick] (0,0) to (6,0);
            \draw[thick,dashed] (2,0.5) to (2,0);
            \node at (0,-0.3) {$\widetilde{a}_k$};
            \node at (0.45,0.56) {$\varphi_{Q_k}^{-1}f(\ell^k)$};
            \node at (2.15,0.25) {{\red $h$}};
            \node at (1.5,-0.1) {{\red $d_1$}};
            \node at (3.5,-0.2) {{\red $d_2$}};
        \end{tikzpicture}
        \caption{Lengths of horizontal slices of $\varphi_{Q_k}^{-1}f(R_k)$, where $h$, $d_1$, and $d_2$ denote the lengths of the corresponding segments (color online)}
        \label{fig:thetriangle}
    \end{figure}
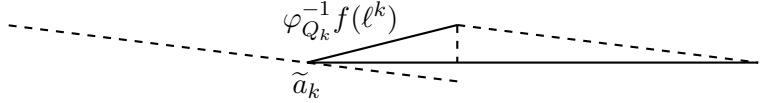
    
    If Condition A holds for $t_0=1$ and some suitable $j_0$ then we are done. Otherwise, there exists some $j_*$ such that $\widetilde{E}_{1}\cap(\R\times\{\frac{j_*}{m}\})\neq\varnothing$. As above, $|\varphi_{Q_k}^{-1}f(R_k)\cap(\R\times\{\frac{j_*}{m}\})|\leq Cn^{-p}\lambda$, picking $k$ large at the beginning, it is not hard to check that Condition A holds for $t_*$ and $j_*$, where $t_*:= \min\{t\geq 2: \widetilde{E}_t\cap(\R\times\{\tfrac{j_*}{m}\})=\varnothing\}$.
\end{proof}

\begin{remark}\label{rem:deindwork}
    Note that in the proof of Cases 2 and 3, we do not need the independence assumption $n\nsim m$ but only the fact that $n>m$. This will save us much effort in the dependent self-affine case.
\end{remark}

\section{The dependent self-affine case}\label{sec:dep}

In this section, we fix a Bedford-McMullen carpet $K=K(n,m,\Lambda)$ with $n>m\geq 2$, $n\sim m$ and assume that $K$ is not supported on any line. Without loss of generality, write $n=m^{p/q}$, where $p,q\in\Z^+$ are coprime. It follows from $n>m$ that $p>q$. Recall that $N:=\max_{j} \#I_j$.

\subsection{Proof of Theorem~\ref{thm:main1}: the dependent case}

Unlike the independent case, when $n\sim m$ we only have information about the dimension of the projections of $K$ in almost all directions, as indicated by Marstrand's projection theorem. This limitation invalidates the argument used in Case 1 of Theorem~\ref{thm:notoblique}. To address it, we employ a rescaling approach that bypasses the need for a projection theorem altogether.

\begin{lemma}\label{lem:dep1}
    Let $f\in\S(\R^2)$ be a contracting map. If $f(K)\subset K$, then we can find $g\in\mathcal{S}(\R^2)$ and a compact set $C\subset\R$ with $\dimh C \leq \frac{\log N}{\log n}$ such that $g(K)\subset C\times\R$. Moreover, if $f$ is oblique, then $g$ can be picked oblique as well.
\end{lemma}
The idea is simple. Rescaling the embedding $f^t(K)$ (as $t\to\infty$) produces a limiting similitude that sends the carpet into a set of the form $C\times\R$, where $C$ is roughly a finite union of horizontal slices. If the original embedding was oblique, the limiting similitude inherits this property, since the orthogonal parts of the rescaled maps either repeat infinitely often or accumulate on an oblique rotation.
\begin{proof}
    Let $\lambda$ be the contraction ratio of $f$. For all positive integers $t$ with $\lambda^t\leq m^{-1}$, define 
    \begin{equation}\label{eq:ktinlemma4.1}
        k_t := \max\{k\geq 0: m^{-pk-1}\geq \lambda^t\}.
    \end{equation}
    Since $f^t$ has contraction ratio $\lambda^t$, and because $K\subset[0,1]^2$ and $f^t(K)\subset K$, $f^t(K)$ is contained in at most four squares of side length $m^{-pk_t}=n^{-qk_t}$. More precisely, there are $x\in n^{-qk_t}\mathbb{N}, y\in m^{-pk_t}\mathbb{N}$ such that 
    \[
        f^t(K) \subset (x,x+2n^{-qk_t}) \times (y,y+2m^{-pk_t}) =: Q_t.
    \]
    Denote by $h_t$ the homothety sending $(0,1)^2$ onto the square 
    $(x,x+n^{-qk_t}) \times (y,y+m^{-pk_t})=:Q'_t$, the ``lower left square'' of $Q_t$. Note that for all $t$,
    \[ 
        h_t^{-1}f^t(K) \subset h_t^{-1}(Q_t) = (0,2)^2.
    \]
    The contraction ratio of $h_t^{-1}f^t$ equals $n^{qk_t}\lambda^t$, which is at most $n^{qk_t}\cdot m^{-pk_t-1}=m^{-1}$ and at least $n^{qk_t}\cdot m^{-p(k_t+1)-1}=m^{-p-1}$. So the sequence $\{h_t^{-1}f^t\}_{t}$ has accumulation points in $\mathcal{S}(\R^2)$. For simplicity, assume that $h_t^{-1}f^t$ converges to some limit $g$.

    Suppose $g(K)\cap(0,1)^2\neq\varnothing$. Then $h_t^{-1}f^t(K)\cap(0,1)^2\neq\varnothing$ for all large $t$.
    For each of these $t$, write $y=\sum_{k=1}^{pk_t}y_km^{-k}$, where $y_k\in\{0,\ldots,m-1\}$. Since $p>q$, the collection of all level-$pk_t$ rectangles in $\overline{Q'_t}$, denoted by $\mathcal{R}_t$, is
    \begin{equation*}
        \Big\{ \Big[ x+\sum_{k=qk_t+1}^{pk_t} \frac{x_k}{n^k}, \Big( x+\sum_{k=qk_t+1}^{pk_t} \frac{x_k}{n^k} \Big)+n^{-pk_t} \Big] \times [y,y+m^{-pk_t}]: (x_k,y_k)\in\Lambda \Big\}.
    \end{equation*} 
    Then we have
    \begin{align}
        h_t^{-1}f^t(K) \cap (0,1)^2 &= h_t^{-1}f^t(K) \cap h_t^{-1}(Q'_t) \notag \\
        &= h_t^{-1}(f^t(K)\cap Q'_t) \notag \\
        &\subset h_t^{-1}\Big( \bigcup\mathcal{R}_t \Big) \label{eq:rtisnotempty} \\
        &\subset \bigcup_{\substack{(x_k,y_k)\in\Lambda, \notag \\ qk_t+1\leq k\leq pk_t}} \Big[ \sum_{k=qk_t+1}^{pk_t} \frac{x_k}{n^{k-qk_t}}, n^{-(p-q)k_t}+ \sum_{k=qk_t+1}^{pk_t} \frac{x_k}{n^{k-qk_t}} \Big] \times\R \notag \\
        &=: E_t \times\R. \label{eq:htinversesupset}
    \end{align}

    Moreover, write $\j_t:=y_{qk_t+1}\cdots y_{pk_t}$, which is a word of length $(p-q)k_t\to\infty$ as $t\to\infty$. Using Cantor's diagonal argument, it is not hard to find an infinite word $\j=j_1j_2\cdots\in J^\infty$ such that for all $M\geq 1$, there is a large $t$ for which the word $\j_t\wedge\j$---the longest common prefix of $\j_t$ and $\j$---has length at least $M$. Define $z:=\sum_{k=1}^\infty j_km^{-k}$.

    We claim that 
    \begin{equation}\label{eq:productstructure}
        g(K)\cap(0,1)^2 \subset K^{z} \times \R.
    \end{equation}
    Fix an arbitrary $\delta>0$ and choose $M\in\mathbb{Z}^+$ so large that $n^{-M}<\frac{\delta}{2}$. Pick $t$ sufficiently large to satisfy the following two conditions:
    \begin{enumerate}
        \item $g(K)\cap (0,1)^2 \subset \mathcal{N}_{\delta/2}(h_t^{-1}f^t(K)\cap(0,1)^2)$, where $\mathcal{N}_{\delta/2}(\cdot)$ denotes the $\delta/2$-neighborhood.
        \item there is a word $\j_t$ such that $\j_t\wedge\j$ has length at least $M$.
    \end{enumerate}  
    By construction of the sets $E_t$ and the definition of $\j_t$, the second condition implies that every point of $E_t$ has its $x$-coordinate within $n^{-M}$ distance of $K^z$, that is, $E_t\subset \mathcal{N}_{n^{-M}}(K^z)$.  
    Hence
    \begin{align*}
        g(K)\cap(0,1)^2 &\subset \mathcal{N}_{\delta/2}(h_t^{-1}f^t(K)\cap(0,1)^2) &&\text{(condition (1))} \\
        &\subset \mathcal{N}_{\delta/2}(E_t\times\R) && \text{(by~\eqref{eq:htinversesupset})} \\
        &\subset \mathcal{N}_{n^{-M}+\delta/2}(K^z\times\R) \subset \mathcal{N}_{\delta}(K^z\times\R).
    \end{align*}
    Since $\delta$ is arbitrary and $K^z$ is closed, we establish~\eqref{eq:productstructure}.
    
    Applying the same argument to the other three squares comprising $Q_t$, we find at most four points $z_{0,0}:=z,z_{0,1},z_{1,0},z_{1,1}\in\pi_2(K)$ such that 
    \begin{equation}\label{eq:formofc}
        g(K) \subset \bigcup_{0\leq i,j\leq 1} (K^{z_{i,j}}+i)\times\R.
    \end{equation}
    Set $C:=\bigcup_{0\leq i,j\leq 1} (K^{z_{i,j}}+i)$. By~\eqref{eq:dimlbofky}, $\dimh C\leq\frac{\log N}{\log n}$.

    It remains to prove that if $f$ is oblique then $g$ can be chosen oblique as well. Write $O$ to be the orthogonal part of $f$. If $\{O^t\}_{t=1}^\infty$ is a finite collection, each element occurs infinitely often. Picking an oblique transformation $\widetilde{O}\in\{O^t\}_t$ and a subsequence $t_k$ where $O^{t_k}=\widetilde{O}$ for all $k$, we can consider $f^{t_k}$ instead of $f^t$ at the beginning of the above arguments and get an oblique $g$ because $g$, as the limit of $h_{t_k}^{-1}f^{t_k}$, has $\widetilde{O}$ as its orthogonal part. If $\{O^t\}_{t=1}^\infty$ is an infinite collection, then since a $2\times 2$ orthogonal matrix is either a rotation or a reflection, $O$ must be an irrational rotation (i.e., the rotation angle is an irrational multiple of $\pi$) and hence every rotation matrix is an accumulation point of $\{O^t\}_{t=1}^\infty$. We could then select a subsequence $\{O^{t_k}\}_k$ so that $O^{t_k}$ converges to an oblique transformation $O'$. Passing to this subsequence in the definition of $g$ gives a limit whose orthogonal part is $O'$, and thus $g$ is oblique. This completes the proof.
\end{proof}

\begin{remark}\label{rem:bairetofindint}
    In the above proof, if we apply Baire's theorem to~\eqref{eq:formofc}, some $(K^{z_{i,j}}+i)\times\R$ should contain an interior part of $g(K)$. In particular, there is an integer $k$ and a level-$k$ rectangle $R$ such that $g\varphi_R(K)\subset (K^{z_{i,j}}+i)\times\R$. As a consequence,
    \begin{equation}\label{eq:fvarphirkyinkz}
        g\varphi_R(K^y\times\{y\}) \subset (K^{z_{i,j}}+i)\times\R, \quad \forall y\in\pi_2(K).
    \end{equation}

    We also remark that the integer $k_t$ defined in~\eqref{eq:ktinlemma4.1} equals $\lfloor (-\frac{t\log\lambda}{\log m}-1)/p\rfloor$. It follows that if $\frac{\log\lambda}{\log n}\notin\Q$, then $\frac{\log\lambda}{\log m}\notin\Q$ as well and hence the sequence $\{n^{qk_t}\lambda^t\}_t$ is dense in an open subinterval of $(0,1)$. Since the proof of Lemma~\ref{lem:dep1} works for any convergent subsequence of $\{h_t^{-1}f^t\}_{t=1}^\infty$, we may select an accumulation map $\widetilde{g}$ so that $\widetilde{g}(K)\subset C_{\widetilde{g}}\times\R$, where $C_{\widetilde{g}}$ is a closed set with $\dimh C_{\widetilde{g}}\leq\frac{\log N}{\log n}$ (as with $C$ in Lemma~\ref{lem:dep1}), and the contraction ratio of $\widetilde{g}$ is an irrational power of $n$. As argued in Section~\ref{subsec:proofofthecom}, this additional freedom in the choice of ratio leads to a contradiction in the proof of Theorem~\ref{thm:main2}. 
\end{remark}

Now we are ready to prove the non-obliqueness theorem for the dependent case. The proof is divided into two parts, depending on whether $f$ is a strict contraction or an isometry. Before proceeding, we outline the intuition for the contracting case (the other case is similar).

Suppose for contradiction that an oblique embedding exists. The previous lemma then yields an oblique similitude $g$ sending $K$ into a strip of the form $C\times\R$, where $C$ is a Cantor set. Based on Corollary~\ref{cor:leftrightendpt}, the similitude $g$ maps some one-sided isolated point of a maximal horizontal slice to a one-sided isolated point of $C$. Since $g$ is oblique, it tilts the level-$k$ rectangle containing that point. But such a tilt forces part of the image rectangle to protrude horizontally beyond the boundary of $C$, thereby pushing some points of the carpet outside $C\times\R$ and leading to a contradiction.

\begin{theorem}\label{thm:dependentaffine}
    If $f$ is a similitude sending $K$ into itself, then $f$ is not oblique.
\end{theorem}
\begin{proof}[Proof of Theorem~\ref{thm:dependentaffine}: the contracting case]
    By Proposition~\ref{prop:noobliquelines}, it suffices to consider when $N\leq n-1$. 
    Suppose on the contrary that $f$ is oblique. Pick $g, C$ as in Lemma~\ref{lem:dep1} with $g$ oblique and let $\rho$ denote the contraction ratio of $g$. If $N=1$, then every horizontal slice of $K$ is a finite set. From the structure of $C$ (see~\eqref{eq:formofc}), it follows that $C$ is finite as well. Since $\#\Lambda\geq 2$ but $N=1$, there exists a pair of selected rectangles lying in different rows. However, by Proposition~\ref{prop:projofkisinf} and the obliqueness of $g$, the projection $\pi_1(g(K))$ is an infinite set. This contradicts the inclusion $\pi_1(g(K))\subset \pi_1(C\times\R)= C$. 
    
    Now suppose $2\leq N\leq n-1$. By Remark~\ref{rem:bairetofindint} (see~\eqref{eq:fvarphirkyinkz}), we can find $z\in\pi_2(K)$, $i\in\Z$, $k\geq 1$ and a level-$k$ rectangle $R$ such that 
    \begin{equation}\label{eq:origvarphir}
        g\varphi_R(K^y\times\{y\}) \subset (K^z+i)\times\R, \quad \forall y\in\pi_2(K).
    \end{equation}
    Since $g$ is a similitude, there is a unit directional vector $v:=(v_1,v_2)$ such that for every $y\in\pi_2(K)$, $g\varphi_R$ maps the horizontal line $\R\times\{y\}$ to $L_y:=\{rv+c_y:r\in\R\}$ for some $c_y\in\R^2$. Together with~\eqref{eq:origvarphir}, we have 
    \begin{equation}\label{eq:gky0slice}
        g\varphi_R(K^{y}\times\{y\}) \subset L_y \cap ((K^z+i)\times \R),\quad \forall y\in\pi_2(K).
    \end{equation}
    Since $g$ is oblique, $v_1,v_2$ are both nonzero. Without loss of generality, assume that $v_1,v_2>0$ (other cases can be similarly discussed). 

    Note that for each $y\in\pi_2(K)$, $g\varphi_R(K^y\times\{y\})$ can be regarded as a similar copy of $K^y$ contracted by $n^{-k}\rho$. Meanwhile, $L_y\cap((K^z+i)\times\R)$ can be viewed as a translated copy of $\frac{|v|}{v_1}K^z$. In particular, from~\eqref{eq:gky0slice} one can derive a similitude $h_y\in\S(\R)$ with contraction ratio $n^{-k}\rho\cdot\frac{v_1}{|v|}$ (which is independent of $y$) such that 
    \begin{equation}\label{eq:hykysubsetc}
        h_y(K^y)\subset K^{z}, \quad \forall y\in\pi_2(K).
    \end{equation}
    
    Choose $j_*\in J$ with $\#I_{j_*}=N$ and define $y_*:=\sum_{k=1}^\infty j_*m^{-k}$. Then $K^{y_*}=E(n,I_{j_*})$ (see~\eqref{eq:expressionofky}). By~\eqref{eq:hykysubsetc}, $h_{y_*}(K^{y_*})\subset K^{z}$. So $0<\h^\alpha(K^{z})<\infty$, where $\alpha:=\frac{\log N}{\log n}$ (recall the observation before Lemma~\ref{lem:bunchoffacts}). It may be beneficial to record an observation here (which will be of help later though not in this proof): since~\eqref{eq:hykysubsetc} also implies $h_z(K^z)\subset K^z$, it follows from Proposition~\ref{prop:allnthenrational} that 
    \begin{equation}\label{eq:thenewlogformula}
        \frac{\log(n^{-k}\rho\cdot v_1/|v|)}{\log n}\in\Q.
    \end{equation}
    
    Next, since $h_z(K^z)\subset K^z$, by Corollary~\ref{cor:leftrightendpt}, $h_z$ maps a left isolated point $a$ and a right one $b$ both to one-sided isolated points of $K^z$. The inclusion $h_z(K^z)\subset K^z$ is simply a one-dimensional reformulation of $g\varphi_R(K^z\times\{z\})\subset L_z\cap ((K^z+i)\times\R)$. This in turn gives us two points $\widetilde{a},\widetilde{b}\in g\varphi_R(K^{z}\times\{z\})$ and some $\delta>0$ such that both tubes $(\pi_1(\widetilde{a})- \delta,\pi_1(\widetilde{a})) \times\R$ and $(\pi_1(\widetilde{b}),\pi_1(\widetilde{b})+\delta)\times\R$ do not intersect $(K^{z}+i)\times\R$. Moreover, there are two level-$t$ rectangles $R_t,R'_t\subset R$ such that $\widetilde{a}\in g\varphi_{R_t}(K)$ and $\widetilde{b}\in g\varphi_{R'_t}(K)$, respectively, and 
    (again by Corollary~\ref{cor:leftrightendpt}) $\pi_2(R_t)=\pi_2(R'_t)$ for all $t\geq k_0$ for some $k_0$. See Figure~\ref{fig:leftrightendpt} for an illustration.

    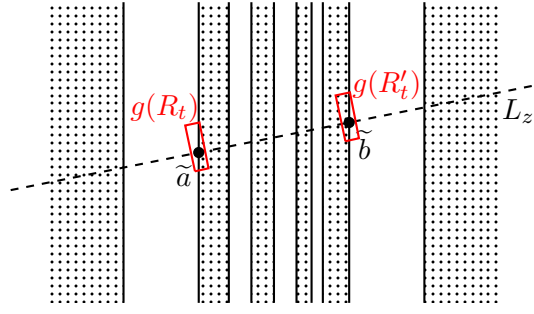
\begin{figure}[htbp]
        \centering
        \begin{tikzpicture}[scale=1]
            \draw[thick] (0,0) to (0,4);
            \draw[thick] (4,0) to (4,4);
            \fill[pattern=dots] (-1,0) rectangle (0,4);
            \fill[pattern=dots] (4,0) rectangle (5,4);
            \draw[thick] (1,0) to (1,4);
            \draw[thick] (1.4,0) to (1.4,4);
            \fill[pattern=dots] (1,0) rectangle (1.4,4);
            \draw[thick] (1.7,0) to (1.7,4);
            \draw[thick] (2,0) to (2,4);
            \fill[pattern=dots] (1.7,0) rectangle (2,4);
            \draw[thick] (2.3,0) to (2.3,4);
            \draw[thick] (2.5,0) to (2.5,4);
            \fill[pattern=dots] (2.3,0) rectangle (2.5,4);
            \draw[thick] (2.65,0) to (2.65,4);
            \draw[thick] (3,0) to (3,4);
            \fill[pattern=dots] (2.65,0) rectangle (3,4);
            \node at (1,2)[circle,fill, inner sep=1.5pt]{};
            \node at (3,2.4)[circle,fill, inner sep=1.5pt]{};
            \draw[thick,dashed] (-1.5,1.5) to (5.5,2.9);
            \node at (0.8,1.7) {$\widetilde{a}$};
            \node at (3.2,2.1) {$\widetilde{b}$};
            \node at (5.25,2.55) {$L_z$};
            \draw[thick,red] (0.95,1.75) to (1.13,1.79) to (1.01,2.4) to (0.82,2.36) to (0.95,1.75);
            \draw[thick,red] (2.95,2.15) to (3.13,2.19) to (3.01,2.8) to (2.82,2.76) to (2.95,2.15);
            \node at (0.55,2.6) [red] {$g(R_t)$};
            \node at (3.5,2.9) [red] {$g(R'_t)$};
        \end{tikzpicture}
        \caption{An illustration of $\widetilde{a},\widetilde{b},g(R_t)$ and $g(R'_t)$, where $(K^z+i)\times\R$ is supported in the shaded region. (color online)}
        \label{fig:leftrightendpt}
    \end{figure}

    Pick $x,x'\in K$ with $\pi_2(x)=\max\pi_2(K)$ and $\pi_2(x')=\min\pi_2(K)$. 
    Since $\widetilde{a}\in g\varphi_{R_t}(K)$, $\widetilde{b}\in g\varphi_{R'_t}(K)$, $x$ can be expressed as
    \begin{equation}\label{eq:expressionofx}
        x = (g\varphi_{R_t})^{-1}(\widetilde{a})+ \Big(\begin{matrix} c_t \\ d_t \end{matrix}\Big) = (g\varphi_{R'_t})^{-1}(\widetilde{b})+\Big(\begin{matrix} \gamma_t \\ \eta_t \end{matrix}\Big)
    \end{equation}
    for some $-1\leq c_t,d_t,\gamma_t,\eta_t\leq 1$. Since $\pi_2(x)$ is the maximum, $d_t,\eta_t\geq 0$. Since $\widetilde{a}$, $\widetilde{b}$ are on the same slice and $\pi_2(R_t)=\pi_2(R'_t)$ for $t\geq k_0$, we have $d_t=\eta_t$ for $t\geq k_0$. Similarly, there are $-1\leq c'_t,\gamma'_t\leq 1$ and $-1\leq d'_t\leq 0$ for all $t\geq k_0$ such that 
    \begin{equation}\label{eq:expressionofxprime}
        x' = (g\varphi_{R_t})^{-1}(\widetilde{a})+\Big(\begin{matrix} c'_t \\ d'_t \end{matrix}\Big) = (g\varphi_{R'_t})^{-1}(\widetilde{b})+\Big(\begin{matrix} \gamma'_t \\ d'_t \end{matrix}\Big).
    \end{equation}

    Fix a large $t\geq k_0$ (will be specified later). Note that
    \[
        d_t-d'_t=\pi_2(x-x')=\max\pi_2(K)-\min\pi_2(K)=:\beta>0.
    \]
    So at least one of $d_t,-d'_t$ is no less than $\beta/2$. 
    Writing $O=(\begin{smallmatrix} o_{1,1} & o_{1,2} \\ o_{2,1} & o_{2,2} \end{smallmatrix})$ to be the orthogonal part of $g$, we have by the obliqueness of $g$ that $o_{1,1},o_{1,2}\neq 0$. We distinguish four simple cases.

    \paragraph*{{\bf Case 1}: $d_t\geq\beta/2$ and $o_{1,2}>0$}
    In this case, we have by~\eqref{eq:expressionofx} that (recall $d_t=\eta_t$)
    \[
        g\varphi_{R'_t}(x) - \widetilde{b} = \Big( \begin{matrix} o_{1,1} & o_{1,2} \\ o_{2,1} & o_{2,2} \end{matrix}\Big)\cdot\Big(\begin{matrix} \rho n^{-t}\gamma_t \\ \rho m^{-t}d_t \end{matrix}\Big).
    \]
    Since $d_t\geq\beta/2$, $n>m$ and $|\gamma_t|\leq 1$, $0<\rho n^{-t}\gamma_to_{1,1}+ \rho m^{-t}d_to_{1,2}<\delta$ when $t$ is large and hence 
    \[
        \pi_1(g\varphi_{R'_t}(x)) =  \pi_1(\widetilde{b})+\rho n^{-t}\gamma_to_{1,1}+\rho m^{-t}d_to_{1,2} \in (\pi_1(\widetilde{b}), \pi_1(\widetilde{b})+\delta).
    \]
    In other words, $g\varphi_{R'_t}(x)$ lies in the tube $(\pi_1(\widetilde{b}),\pi_1(\widetilde{b})+\delta)\times\R$. But we have seen that this tube does not intersect $(K^{z}+i)\times\R\supset g\varphi_R(K)\ni g\varphi_{R'_t}(x)$, which leads to a contradiction.

    \paragraph*{{\bf Case 2}: $d_t\geq\beta/2$ and $o_{1,2}<0$}
    In this case, we have by~\eqref{eq:expressionofx} that 
    \[
        g\varphi_{R_t}(x) - \widetilde{a} = \Big( \begin{matrix} o_{1,1} & o_{1,2} \\ o_{2,1} & o_{2,2} \end{matrix}\Big)\cdot\Big(\begin{matrix} \rho n^{-t}c_t \\ \rho m^{-t}d_t \end{matrix}\Big).
    \]
    Since $d_t\geq\beta/2$, $n>m$ and $|c_t|\leq 1$, we have $-\delta<\rho n^{-t}c_to_{1,1}+ \rho m^{-t}d_to_{1,2}<0$ when $t$ is large and hence 
    \[
        \pi_1(g\varphi_{R_t}(x)) =  \pi_1(\widetilde{a})+ \rho n^{-t}c_to_{1,1}+ \rho m^{-t}d_to_{1,2} \in (\pi_1(\widetilde{a})-\delta, \pi_1(\widetilde{a})).
    \]
    In other words, $g\varphi_{R_t}(x)$ lies in the tube $(\pi_1(\widetilde{a})-\delta,\pi_1(\widetilde{a}))\times\R$. Again, we have seen that this tube does not intersect $(K^{z}+i)\times\R\supset g\varphi_R(K)\ni g\varphi_{R_t}(x)$, which leads to a contradiction.

    \paragraph*{{\bf Case 3}: $-d'_t\geq\beta/2$ and $o_{1,2}>0$}
    Applying the argument in Case 2 to~\eqref{eq:expressionofxprime}, we have  
    \[
        \pi_1(g\varphi_{R_t}(x')) =  \pi_1(\widetilde{a})+ \rho n^{-t}c'_to_{1,1}+\rho m^{-t}d'_to_{1,2} \in (\pi_1(\widetilde{a})-\delta, \pi_1(\widetilde{a}))
    \]
    when $t$ is large, which leads to a contradiction.

    \paragraph*{{\bf Case 4}: $-d'_t\geq\beta/2$ and $o_{1,2}<0$}
    Applying the argument in Case 1 to~\eqref{eq:expressionofxprime}, we have  
    \[
        \pi_1(g\varphi_{R'_t}(x')) =  \pi_1(\widetilde{b})+\rho n^{-t}\gamma'_to_{1,1}+ \rho m^{-t}d'_to_{1,2} \in (\pi_1(\widetilde{b}), \pi_1(\widetilde{b})+\delta)
    \]
    when $t$ is large, which leads to a contradiction.
\end{proof}

The proof for the isometry case proceeds by rescaling the embedded image in a manner analogous to Lemma~\ref{lem:dep1}, with suitable modifications for the isometry assumption.

\begin{lemma}
    Let $f\in\S(\R^2)$ be an oblique isometry. If $f(K)\subset K$, then we can find some $z\in\pi_2(K)$ and a nondegenerate scaled copy $E$ of $E(m,J)$ such that $E\subset K^z$.
\end{lemma}
\begin{proof}
    Choose a nested sequence of level-$k$ rectangles $R_k$ with $R_{k+1}\subset R_k$ for $k\geq 1$. Since $f$ is an isometry, we have 
    \[
        \max\{|\pi_1(f(R_k))|,|\pi_2(f(R_k))|\} < 2m^{-k}, \quad \forall k\geq 1.
    \]
    Recall that $n^q=m^p$. By the above estimate, $f(R_{pk})$ is contained in at most four squares of side length $m^{-pk}=n^{-qk}$. More precisely, there are $x\in n^{-qk}\mathbb{Z}, y\in m^{-pk}\mathbb{Z}$ such that 
    \[
        f\varphi_{R_{pk}}(K) \subset (x,x+2n^{-qk}) \times (y,y+2m^{-pk}) =: Q_k.
    \]
    Denote by $h_k$ the homothety sending $(0,1)^2$ to the square 
    $(x,x+n^{-qk}) \times (y,y+m^{-pk})=:Q'_k$, the lower left square of $Q_k$. 
    
    Since $E(m,J)=\pi_2(K)$, it is easy to see that $d_H(\varphi_{R_{pk}}(K),\varphi_{R_{pk}}(\{0\}\times E(m,J)))\leq n^{-pk}$, where $d_H$ denotes the Hausdorff distance (see~\cite{Fal14}). So
    \begin{equation}\label{eq:tendstotheleftproj}
        d_H(h_k^{-1}f\varphi_{R_{pk}}(K), h_k^{-1}f\varphi_{R_{pk}}(\{0\}\times E(m,J))) \leq n^{(q-p)k} \to 0, \quad k\to\infty.
    \end{equation}
    Moreover, writing $O$ to be the orthogonal part of $f$, each $h_k^{-1}f\varphi_{R_{pk}}(\{0\}\times E(m,J))$ is simply a translated copy
    of $O(\{0\}\times E(m,J))$, say $O(\{0\}\times E(m,J))+a_k$. 
    So~\eqref{eq:tendstotheleftproj} becomes 
    \[
        d_H(h_k^{-1}f\varphi_{R_{pk}}(K), O(\{0\}\times E(m,J))+a_k) \leq n^{(q-p)k} \to 0, \quad k\to\infty.
    \]
    By the definition of $h_k$, the sequence $\{a_k\}_k$ lies in a bounded region in the plane. For notational simplicity, let us assume without loss of generality that $\{a_k\}_k$ is convergent, say $a_k\to a$.  Then under the Hausdorff distance,
    \[
        h_k^{-1}f\varphi_{R_{pk}}(K) \to O(\{0\}\times E(m,J))+a, \quad k\to\infty.
    \]

    Applying the rescaling argument in the proof of Lemma~\ref{lem:dep1} (now to $\{h_k^{-1}f\varphi_{R_{pk}}\}_k$ instead of $\{h_t^{-1}f^t\}_t$),
    we can find at most four points $z_{i,j}\in\pi_2(K)$, $0\leq i,j\leq 1$, such that 
    \[
        O(\{0\}\times E(m,J))+a \subset \bigcup_{0\leq i,j\leq 1} (K^{z_{i,j}}+i)\times\R.
    \]
    Since $O$ is oblique, projecting the above two sets to the $x$-axis, we see that $\bigcup_{0\leq i,j\leq 1} (K^{z_{i,j}}+i)$ contains a nondegenerate scaled copy of $E(m,J)$. Then by Baire's theorem, some $K^{z_{i,j}}+i$ must contain an interior part of that copy of $E(m,J)$. Since $E(m,J)$ is self-similar, this completes the proof.
\end{proof}

\begin{proof}[Proof of Theorem~\ref{thm:dependentaffine}: the isometry case]
    By Remark~\ref{rem:deindwork} and Proposition~\ref{prop:noobliquelines}, it suffices to consider when $N\leq n-1$ and $\#J\leq m-1$. Since $K$ is not contained in any line, $\#J>1$. Thus $E(m,J)$ is a deleted-digit set with Hausdorff dimension $\frac{\log\#J}{\log m}\in(0,1)$.

    Assume on the contrary that $f$ is oblique. Note that 
    \[
        f(K^y\times\{y\}) \subset f(K)\subset K \subset \R\times\pi_2(K) = \R\times E(m,J), \quad \forall y\in\pi_2(K).
    \]
    Therefore, $\pi_2(f(K^y\times\{y\})) \subset E(m,J)$ for all $y$. Since $f$ is oblique, the left hand side is a nondegenerate scaled copy of $K^y$. Picking $y_*$ as in the proof of Theorem~\ref{thm:dependentaffine} (see the discussion before~\eqref{eq:thenewlogformula}), we have 
    \[
        \dimh E(m,J)\geq\dimh K^{y_*}=\frac{\log N}{\log n}=:\alpha.
    \]
    
    Now apply the lemma above to obtain a point $z\in\pi_2(K)$ such that $K^z$ contains a scaled copy of $E(m,J)$. So there are $\beta_2,c_2\in\R$ with $E(m,J)\subset\beta_2K^z+c_2$. Conversely, we have seen that $E(m,J)$ also contains a scaled copy of $K^z$, say $\beta_1K^z+c_1$. Together these inclusions yield $\dimh E(m,J)=\dimh K^z=\alpha$. Since $\alpha>0$, $N\geq 2$. Since $0<\h^\alpha(E(m,J))<\infty$, $0<\h^\alpha(K^z)<\infty$. 

    Finally, note that 
    \[
        f(K^z\times\{z\})\subset\R\times E(m,J)\subset \R\times (\beta_2 K^z+c_2).
    \]
    Using this inclusion together with Corollary~\ref{cor:leftrightendpt}, we may argue exactly as in the second part of the proof of Theorem~\ref{thm:dependentaffine} (the part following~\eqref{eq:thenewlogformula}) to reach a contradiction. The only difference is that whereas previously a slice was embedded obliquely into a union of vertical lines, here it is embedded obliquely into a union of horizontal lines.
\end{proof}

\subsection{Proof of Theorem~\ref{thm:main2}}
\label{subsec:proofofthecom}

With the non-obliqueness statement established above and Proposition~\ref{prop:allnthenrational} in hand, Theorem~\ref{thm:main2} follows quickly.

\begin{proof}[Proof of Theorem~\ref{thm:main2}]
    Let $K,n,m,f$ be as in Theorem~\ref{thm:main2}. If $\lambda=1$, there is nothing to prove. Assume that $\lambda<1$. By Theorem~\ref{thm:dependentaffine}, $f$ cannot be oblique. So there exists an integer $k\geq 1$ such that $f^k(x)=\lambda^kx+\eta$ for some $\eta\in\R^2$, that is, the orthogonal part of $f^k$ is the identity matrix. As a consequence, $f^k$ sends the $x$-axis to a  horizontal line. 
    
    Suppose on the contrary that $\frac{\log\lambda}{\log n}\notin\Q$. Applying Lemma~\ref{lem:dep1} to $f^k$ and recalling the second paragraph in Remark~\ref{rem:bairetofindint}, we obtain some $g\in\S(\R^2)$ of which the contraction ratio $\rho$ is an irrational power of $n$. Also, since $f^k$ involves neither rotations nor reflections, the same holds for $g$. 
    
    If $2\leq N\leq n-1$, we may repeat the argument used in the proof of the contracting case of Theorem~\ref{thm:notoblique} (now with direction vector $v=(1,0)$), which leads again to~\eqref{eq:thenewlogformula} and hence $\frac{\log\rho}{\log n}\in\Q$, a contradiction.

    If $N=1$, then every row contains at most one selected rectangle in the initial pattern of $K$. In particular, $\#\Lambda= \#J\leq m<n$. So there must be at least one vacant column. On the other hand, since $\#\Lambda\geq 2$ and $K$ is not contained in any vertical line, we have $\#I\geq 2$. Hence $\pi_1(K)=E(n,I)$ is a deleted-digit set with Hausdorff dimension $0<\frac{\log\#I}{\log n}\leq\frac{\log\#\Lambda}{\log n}<1$. Since $f^k(K)\subset K$, $\pi_1(f^k(K)) \subset \pi_1(K)=E(n,I)$. So 
    \[
        \lambda^kE(n,I)+\pi_1(\eta) = \lambda^k\pi_1(K)+\pi_1(\eta) = \pi_1(f^k(K)) \subset E(n,I).
    \] 
    Then it follows from Lemma~\ref{lem:fw09} that $\frac{\log\lambda}{\log n}\in\Q$. 
    
    Finally, assume that $N=n$. Let $J':=\{j\in J: \#I_{j}=n\}$. Note that 
    \begin{equation}\label{eq:ycontainsintervals}
        \{y\in\pi_2(K): K^y \text{ contains an interval}\} \subset \bigcup_{t\geq 0}\bigcup_{z\in\Z} \frac{E(m,J')+z}{n^t}.
    \end{equation} 
    The case when $\#J'=1$ has already been settled in~\cite[Section 6.6]{AH19}. If $\#J'\geq 2$ (clearly, $\#J'\leq m-1$), then $0<\dimh E(m,J')<1$. Note that for every $y\in E(m,J')$, $K^y$ is an interval of length $1$. Since $f^k(K)\subset K$, it follows from~\eqref{eq:ycontainsintervals} that 
    \[
        \lambda^k E(m,J')+\pi_2(\eta) \subset \bigcup_{t\geq 0}\bigcup_{z\in\Z} \frac{E(m,J')+z}{n^t}.
    \]
    By Baire's theorem, there exist $t,z$ such that $\frac{E(m,J')+z}{n^t}$ contains an interior part of the left hand side. Applying Lemma~\ref{lem:fw09} once more yields $\frac{\log\lambda}{\log n}\in\Q$, completing the proof.
\end{proof}

\section{The dependent self-similar case}\label{sec:selfsim}

Finally, we consider the case when $m=n$, that is, $K=K(n,\Lambda)$ is a generalized Sierpi\'nski carpet. However, there exists such a self-similar carpet admitting oblique embeddings.

\begin{example}\label{exa:nonobsiercat}
    Consider the carpet 
    \[
        K = \Big\{ \sum_{k=1}^\infty \frac{(x_k,y_k)}{4^k}: (x_k,y_k)\in\{(0,1),(1,3),(2,0),(3,2)\} \Big\}.
    \]
    See Figure~\ref{fig:specialcarpet}.
    Let $O=\big(\begin{smallmatrix} \frac{3}{5} & -\frac{4}{5}\\ -\frac{4}{5} & -\frac{3}{5} \end{smallmatrix}\big)$, which is the matrix for the reflection with respect to some line of slope $-1/2$. It is easy to check that 
    \begin{align*}
        OK &+ ( \tfrac{9}{5}, \tfrac{18}{5} ) \\
        &= \Big\{ \sum_{k=1}^\infty \frac{(\frac{3}{5}x_k-\frac{4}{5}y_k+\frac{9}{5},-\frac{4}{5}x_k-\frac{3}{5}y_k+\frac{18}{5})}{4^k}: (x_k,y_k)\in\{(0,1),(1,3),(2,0),(3,2)\} \Big\} \\
        &= \Big\{ \sum_{k=1}^\infty \frac{(x_k,y_k)}{4^k}: (x_k,y_k)\in\{(0,1),(1,3),(2,0),(3,2)\} \Big\} = K.
    \end{align*}
    Therefore, $K$ admits an oblique self-embedding $f(x)=Ox+(\frac{9}{5},\frac{18}{5})$.
    \begin{figure}[htbp]
        \centering
        \includegraphics[width=3.8cm]{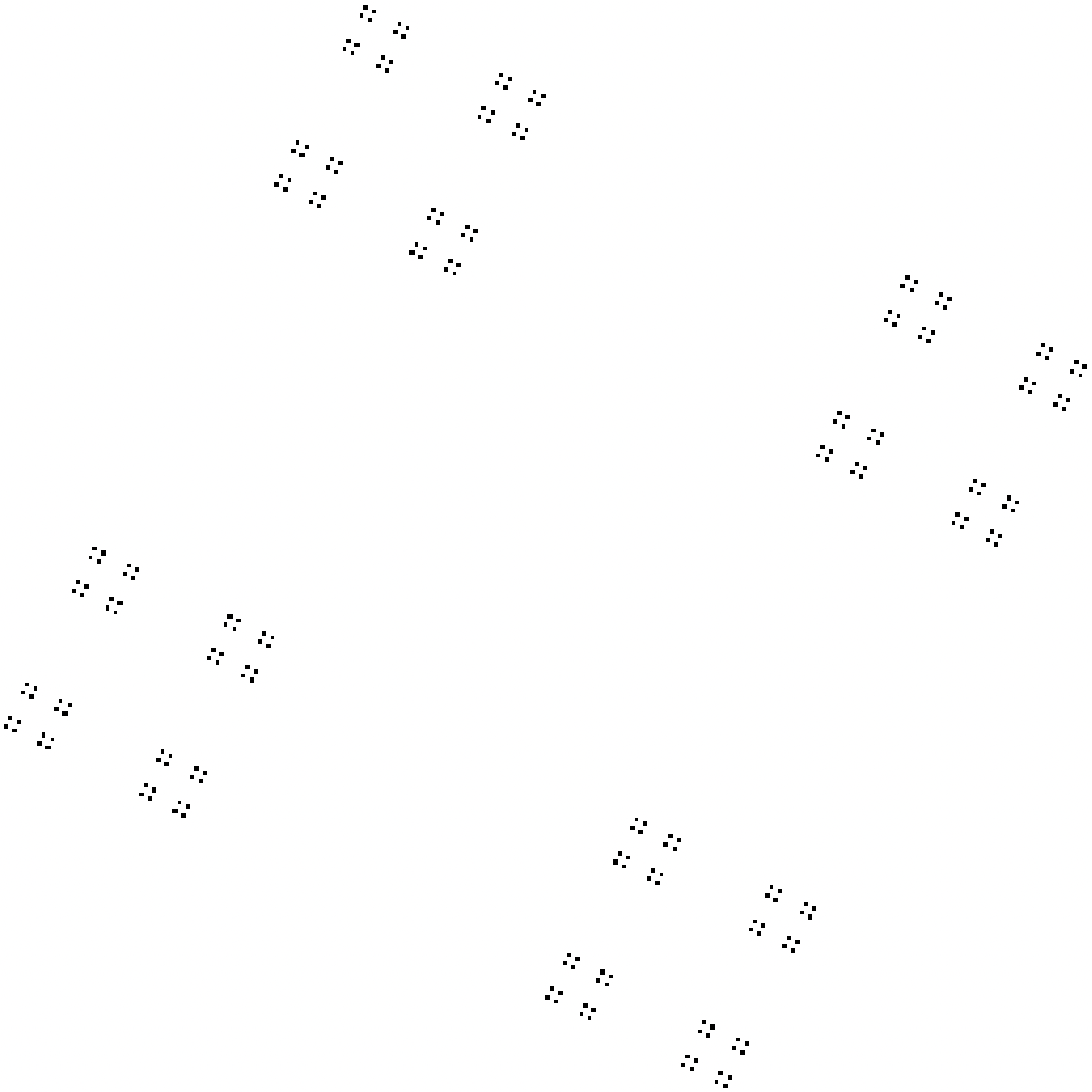}
        \caption{A self-similar carpet admitting oblique self-embeddings}
        \label{fig:specialcarpet}
    \end{figure}
\end{example}

In the rest of this section, we only consider rotational self-embeddings of generalized Sierpi\'nski carpets and hope to extract information about the rotation angle. Recall the notation~\eqref{eq:varphiij}. The following result is well known (see e.g.~\cite{SW99}).

\begin{lemma}\label{lem:convandfixpt}
    If $K$ is not contained in any line, then the convex hull of $K$ is a polygon. Moreover, every vertex of this polygon is the fixed point of $\varphi_{i,j}$ for some $(i,j)\in\Lambda$.
\end{lemma}

From now on, let us fix a generalized Sierpi\'nski carpet $K$ satisfying the strong separation condition and not supported in any line. Denote by $P$ the convex hull of $K$. By the above lemma, $P$ is a polygon with finitely many vertices, say $v_1,\ldots,v_p$. For $1\leq t\leq p$, we write $\alpha_t$ to be the interior angle of $P$ at the vertex $v_t$, and write $\ell_t$ to be the edge of $P$ joining $v_t$ and $v_{t+1}$ (with indices taken cyclically). We claim that every $v_t$ is not isolated in $K$ even along its adjacent edges.

\begin{corollary}\label{cor:limitptonell}
    For $1\leq t\leq p$, there exists $\{x_k\}_{k=1}^\infty\subset\ell_t\cap K$ such that $x_k\to v_t$ as $k\to\infty$. 
\end{corollary}
\begin{proof}
    Fix any $1\leq t\leq p$. By Lemma~\ref{lem:convandfixpt}, there is some $(i,j)\in\Lambda$ such that $v_t$ is the fixed point of $\varphi_{i,j}$. Then 
    \[
        v_t -\varphi_{i,j}^k(v_{t+1}) = \varphi_{i,j}^k(v_{t})-\varphi_{i,j}^k(v_{t+1}) = n^{-k}(v_t-v_{t+1}), \quad \forall k\geq 1.
    \]
    In particular, $\varphi_{i,j}^k(v_{t+1})$ lies on the edge $\ell_t$ and converges to $v_t$ as $k\to\infty$.
\end{proof}

\begin{corollary}\label{cor:rationaltangent}
    For each $1\leq t\leq p$, $\tan\alpha_t\in\Q$.
\end{corollary}
\begin{proof}
    For $1\leq t\leq p$, we again find by Lemma~\ref{lem:convandfixpt} some $(i_t,j_t)\in\Lambda$ such that $v_t$ is the fixed point of $\varphi_{i_t,j_t}$. Therefore, $v_t=(\frac{i_t}{n-1},\frac{j_t}{n-1})$. This implies that the angle made by every edge $\ell_t$ and the $x$-axis has a rational tangent. It follows directly that $\tan\alpha_t\in\Q$ for all $t$.
\end{proof}

A ``local openness'' result as follows will be of great help.

\begin{lemma}[{\cite[Theorems~4.5 and 4.9]{EKM10}}]\label{lem:ekmtworesults}
    Let $f\in\S(\R^2)$. If $f(K)\subset K$, then there is an open set $U\subset\R^2$ such that $U\cap f(K)=U\cap K\neq\varnothing$. Moreover, $\{O^t\}_{t=1}^\infty$ is a finite collection, where $O$ denotes the orthogonal part of $f$.
\end{lemma}

Due to the local openness property mentioned above, there is a small region where $f(K)$ and $K$ are identical. Within this region, we can find a tiny copy of $K$ that lies entirely in $f(K)$. In particular, this copy must be mapped from another small copy of $K$. By examining how the vertices and edges of the associated convex hulls behave under this transformation, we can align a vertex of one copy with a vertex of another, as well as their adjacent edges. A more precise statement is as follows.

\begin{proposition}\label{prop:rationalangles}
    If $f\in\S(\R^2)$ sends $K$ into itself, then for every $1\leq t\leq p$, we can find $t'$ and two squares $R,R'$ (of level $k$ and $k'$, respectively) such that the following properties hold:
    \begin{enumerate}
        \item $f\varphi_{R'}(v_{t'})=\varphi_R(v_t)$,
        \item either $f\varphi_{R'}(\ell_{t'})\subset\varphi_R(\ell_t)$ or $f\varphi_{R'}(\ell_{t'+1})\subset\varphi_{R}(\ell_{t})$.
    \end{enumerate}
\end{proposition}

\begin{proof}
    Fix $1\leq t\leq p$. Without loss of generality, assume that $v_t$ lies in the ``left bottom area'' of $[0,1]^2$, that is, $v_t\in[0,1/2]^2$ (other cases can be similarly discussed). Let us pick an integer $k$ and a level-$k$ square $R=[\frac{i}{n^k},\frac{i+1}{n^k}]\times[\frac{j}{n^k},\frac{j+1}{n^k}]$ such that $\varphi_R(K)\subset f(K)$ and
    \[
        K \cap \big( (\tfrac{i-1}{n^k},\tfrac{i}{n^k})\times(\tfrac{j-1}{n^k},\tfrac{j+1}{n^k}) \big) = \varnothing \quad\text{and}\quad K \cap \big( (\tfrac{i}{n^k},\tfrac{i+1}{n^k})\times(\tfrac{j-1}{n^k},\tfrac{j}{n^k}) \big) = \varnothing.
    \]
    The existence of $R$ will be proved at the end of the proof. As a consequence, there is a small $r>0$ such that 
    \begin{equation}\label{eq:localareaofphirvt}
        K \cap B(\varphi_R(v_t),r) = \varphi_R(K) \cap B(\varphi_R(v_t),r) = K \cap \varphi_R(P) \cap B(\varphi_R(v_t),r),
    \end{equation}
    where $B(x,r)$ denotes the ball centered at $x$ and of radius $r$. Since $\varphi_R(v_t)\in\varphi_R(K)\subset f(K)$, there is a level-$k'$ square $R'$ (for some large $k'$) such that $\varphi_R(v_t)\in f\varphi_{R'}(K)$. We claim that $\varphi_R(v_t)$ is a vertex of the polygon $f\varphi_{R'}(P)$ if $k'$ is picked large. In particular, there exists $t'$ such that $\varphi_R(v_t)=f\varphi_{R'}(v_{t'})$, which establishes (1). 

    To see the claim, note that since $P$ is the convex hull of $K$, $\varphi_R(v_t)\in f\varphi_{R'}(K)\subset f\varphi_{R'}(P)$. So if $k'$ is large enough, then 
    \begin{equation}\label{eq:fvarphirinball}
        f\varphi_{R'}(P) \subset B(\varphi_{R}(v_t),r).
    \end{equation}
    Therefore,
    \begin{align*}
        f\varphi_{R'}(K) &\subset f(K) \cap f\varphi_{R'}(P) && \text{(since $K\subset P$)} \\
        &\subset K\cap B(\varphi_R(v_t),r) && \text{(by $f(K)\subset K$ and~\eqref{eq:fvarphirinball})} \\
        &\subset \varphi_R(P) &&\text{(by~\eqref{eq:localareaofphirvt})}.
    \end{align*}
    So $f\varphi_{R'}(P)\subset\varphi_R(P)$. Since $f\varphi_{R'}(P)$, $\varphi_R(P)$ are both polygons and $\varphi_R(v_t)\in f\varphi_{R'}(P)$, $\varphi_R(v_t)$ must be a vertex of $f\varphi_{R'}(P)$.
    
    If (2) is false, we see by (1) that $\varphi_{R}(\ell_t)$ meets $f\varphi_{R'}(P)$ only at the vertex $\varphi_R(v_t)$. By Corollary~\ref{cor:limitptonell}, there is a sequence $\{x_p\}_{p=1}^\infty\subset\ell_t\cap K$ such that $x_p\to v_t$ as $p\to\infty$. So $\varphi_R(x_p)\to\varphi_R(v_t) \in f\varphi_{R'}(K)$ and $\varphi_R(x_p)\in\varphi_R(\ell_t)$. Since $\varphi_{R}(\ell_t)$ meets $f\varphi_{R'}(P)$ only at the vertex $\varphi_R(v_t)$, $\varphi_R(x_p)\notin f\varphi_{R'}(K)$. But recalling that $\varphi_R(K)\subset f(K)$, we have $\varphi_R(x_p)\in f(K)$ for all $p$. This implies that 
    \[
        \dist\big( f\varphi_{R'}(K), f(K)\setminus f\varphi_{R'}(K) \big) \leq \inf_p\dist(f\varphi_{R'}(K),\{\varphi_R(x_p)\})= 0, 
    \]
    which contradicts the strong separation condition.

    It remains to prove the existence of $R$. By Lemma~\ref{lem:ekmtworesults}, there is some $k_0$ and a level-$k_0$ square $Q$ such that $\varphi_Q(K)\subset f(K)$.  Since $K$ is not contained in any line, we have $\#I\geq 2$ and $\#J\geq 2$ (here we abuse slightly these notations from Section 2.1). Iterating the initial pattern if necessary, we may assume that there are $(i_1,j_1),(i_2,j_2),(i_3,j_3)\in\Lambda$ such that
    \[
        (i_1,j_1-1), (i_2-1,j_2),(i_3-1,j_3-1) \in \{0,1,\ldots,n-1\}^2\setminus\Lambda.
    \]
    Also, write $\underline{i}:=\min I$ and $\underline{j}:= \min J$. 
    
    If there is some $1\leq i\leq n-1$ such that $(i,\underline{j})\in\Lambda$ but $(i-1,\underline{j})\notin\Lambda$, then it suffices to take $R=\varphi_Q\varphi_{i_1,j_1}\varphi_{i,\underline{j}}([0,1]^2)$. If such a digit does not exist, then $(0,\underline{j})\in\Lambda$. If $I_{\underline{j}}\supsetneq\{0\}$, then $2\leq\#I_{\underline{j}}\leq n-1$, where the upper bound holds because $K$ does not contain horizontal segments. So there are at least $4$ level-$2$ squares in the bottom row and it is not hard to see that at least one of them, say $Q'$, satisfies that the $n^2$-adic square left adjacent to $Q'$ is contained in $[0,1]^2$ but unselected. Then it suffices to take $R=\varphi_Q\varphi_{i_1,j_1}(Q')$. So we may assume that $I_{\underline{j}}=\{0\}$. Similarly (looking at $(i_2,j_2)$ instead), it suffices to consider when $J_{\underline{i}}=\{0\}$. 
    
    Since $(0,\underline{j})\in\Lambda$, $\underline{i}=0$. So $(0,0)\in\Lambda$ but $\{(0,j),(j,0):j\geq 1\}\cap\Lambda=\varnothing$. Then it is easy to see that taking $R=\varphi_Q\varphi_{i_3,j_3}\varphi_{0,0}([0,1]^2)$ would work. 
\end{proof}

\begin{lemma}[Niven's theorem]\label{lem:niven}
    If $\theta\in\pi\Q$ and $\tan\theta\in\Q$, then $\tan\theta\in\{0,\pm 1\}$.
\end{lemma}
\begin{proof}
    For a proof, see~\cite{PV21}.
\end{proof}

\begin{corollary}
    Let $f(x)=\lambda R_\theta x+a$ be as in Theorem~\ref{thm:main3}. If $\frac{\theta}{\pi}\in\Q$ and $f$ is an oblique self-embedding similitude of $K$, then $|\tan\theta|=1$.
\end{corollary}
\begin{proof}
    By Lemma~\ref{lem:niven} and the obliqueness of $f$, it suffices to show that $\tan\theta\in\Q$. Fix any $1\leq t\leq p$ and pick $R,R',k,k',t'$ accordingly as in Proposition~\ref{prop:rationalangles}. Without loss of generality, assume that $f\varphi_{R'}(\ell_{t'})\subset\varphi_{R}(\ell_t)$, which is the first case in Proposition~\ref{prop:rationalangles}(2). 
    
    \begin{figure}[htbp]
        \centering
        \begin{tikzpicture}[rotate=30, scale=0.5,>=stealth]
            \draw[thick] (0,0) to (-2,0);
            \draw[thick,red] (-2,0) to (-3,-2);
            \draw[thick] (-3,-2) to (-2.5,-2.8);
            \draw[thick] (-2.5,-2.8) to (-1.5,-4.8);
            \draw[thick,red] (-1.5,-4.8) to (0,-4.8);
            \node at (-3,-1) [red] {$\ell_{t'}$};
            \node at (-0.5,-5.3) [red] {$\ell_t$};
            \node at (-2,0)[circle,fill, inner sep=1.5pt]{};
            \node at (-2.5,0.3) {$v_{t'}$};
            \node at (-1.5,-4.8)[circle,fill, inner sep=1.5pt]{};
            \node at (-1.8,-5.3) {$v_t$};
            \draw[thick,dashed] (-3,-2) to (-3.8,-3.6);
            \draw[thick,dashed] (-1.5,-4.8) to (-1,-5.8);
            \coordinate (A) at (-2,0);
            \coordinate (B) at (-3,-2);
            \coordinate (C) at (-2.5,-2.8);
            \coordinate (D) at (-1.5,-4.8);
            \coordinate (E) at (0,0);
            \coordinate (F) at (0,-4.8);
            \draw pic["$\alpha_{t'}$", draw=black, -, angle eccentricity=2.3,angle radius=0.2cm]{angle = B--A--E};
            \draw pic["$\alpha_{t'+1}$", draw=black, -, angle eccentricity=2.6,angle radius=0.2cm]{angle = C--B--A};
            \draw pic["$\alpha_{t}$", draw=black, -, angle eccentricity=1.8,angle radius=0.2cm]{angle = F--D--C};
        \end{tikzpicture}
        \caption{The rotation effect of $f$ (color online)}
        \label{fig:rotationoff}
    \end{figure}
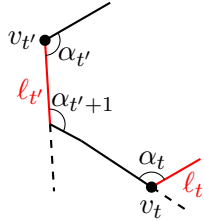

    Since $\varphi_R(\ell_{t})$ is parallel to $\varphi_{R'}(\ell_{t})$ and $\varphi_{R'},\varphi_R$ involve neither rotations nor reflections, $f$ sends a line parallel to $\ell_{t'}$ to another line parallel to $\ell_t$. Without loss of generality, assume that $t'<t$  and $v_1,\ldots,v_p$ are in counterclockwise order. Then a simple geometric observation (see Figure~\ref{fig:rotationoff}) tells us that $R_\theta$ can be realized as follows: first rotate the line containing $\ell_{t'}$ to the one containing $\ell_{t'+1}$, and then to the one containing $\ell_{t'+2}$ and so on until to $\ell_t$. Thus
    \[
        \theta = \sum_{p=1}^{t-t'} (\pi-\alpha_{t'+p}) \pmod{2\pi}.
    \]
    By Corollary~\ref{cor:rationaltangent}, $\tan\theta\in\Q$ as desired.
\end{proof}

With all these observations in hand, Theorem~\ref{thm:main3} can be easily proved.

\begin{proof}[Proof of Theorem~\ref{thm:main3}]
    Again, we may assume that $K$ is not supported in any line. By Lemma~\ref{lem:ekmtworesults}, $\frac{\theta}{\pi}\in\Q$. Then the statement follows directly from the above corollary.
\end{proof}

\begin{remark}
    The arguments developed in this paper seem inadequate to determine whether a rotation angle of $\pi/4$ can occur. It would be interesting to explore whether new geometric observations can be formulated to settle this question---though it follows readily from Lemma~\ref{lem:convandfixpt} that the convex hull of any such carpet cannot be a regular octagon.
\end{remark}

\bigskip
\noindent{\bf Acknowledgements.}
The author is partially supported by NSFC grant No. 12501113 and the Fundamental Research Funds for the Central Universities No. NS2025025. He thanks Professor Huo-Jun Ruan for reading the manuscript carefully and pointing out several mistakes and a number of typos in an earlier version, and Dingkun Hu for a nice observation regarding the setup of Condition A in the proof of Theorem~\ref{thm:notoblique} (under Case 3). He is also grateful to the anonymous referees for their thorough reviews of the manuscript and for providing numerous valuable comments, which significantly improve the presentation of this article.

\small
\bibliographystyle{amsplain}

\begin{thebibliography}{100}

    \bibitem{Alg201}
    A. Algom, Affine embeddings of Cantor sets in the plane, J. Anal. Math. {\bf 140} (2020), 695--757. 

    \bibitem{Alg20}
	A. Algom, Slicing theorems and rigidity phenomena for self-affine carpets, Proc. Lond. Math. Soc. {\bf 121} (2020), 312--353. 

	\bibitem{AH19}
	A. Algom and M. Hochman, Self embeddings of Bedford-McMullen carpets, Ergod. Th. \& Dynam. Sys. {\bf 39} (2019), 577--603.

	\bibitem{AW23}
	A. Algom and M. Wu, Improved versions of some Furstenberg type slicing theorems for self-affine carpets, Int. Math. Res. Not. {\bf 3} (2023), 2304--2343.

    \bibitem{AW25}
    A. Algom and M. Wu, Large slices through self affine carpets, in {\it Recent developments in fractals and related fields}, 1--16, Trends Math., Birkh\"auser/Springer, Cham, 2025.

    \bibitem{BR14}
    B. B\'ar\'any and M. Rams, Dimension of slices of Sierpi\'nski-like carpets, J. Fractal Geom. {\bf 1} (2014), 273--294.


    \bibitem{Bed84}
    T. Bedford, Crinkly curves, Markov partitions and box dimensions in self-similar sets, Ph.D. thesis, University of Warwick, 1984.

    \bibitem{BP17}
    C.~J. Bishop and Y. Peres, {\it Fractals in probability and analysis}, Cambridge Studies in Advanced Mathematics, 162, Cambridge Univ. Press, Cambridge, 2017.


    \bibitem{EKM10}
    M. Elekes, T. Keleti, and A. M\'ath\'e, Self-similar and self-affine sets: measure of the intersection of two copies, Ergod. Th. \& Dynam. Sys. {\bf 30} (2010), 399--440.

    \bibitem{Fal14}
    K. J. Falconer, \textit{Fractal Geometry: Mathematical Foundations and Applications}, 3rd edn. John Wiley \& Sons, Chichester, 2014.

    \bibitem{FHR14}
    D.-J. Feng, W. Huang, and H. Rao, Affine embeddings and intersections of Cantor sets, J. Math. Pures Appl. {\bf 102} (2014), 1062--1079.

    \bibitem{FW09}
    D.-J. Feng and Y. Wang, On the structures of generating iterated function systems of Cantor sets, Adv. Math. {\bf 222} (2009), 1964--1981. 

    \bibitem{FJS10}
    A. Ferguson, T. Jordan, and P. Shmerkin, The Hausdorff dimension of the projections of self-affine carpets, Fund. Math. {\bf 209} (2010), 193--213.

    \bibitem{Fra21}
    J. Fraser, Fractal geometry of Bedford-McMullen carpets, In M. Pollicot and S. Vaienti, editors, Proceedings of the Fall 2019 Jean-Morlet Chair programme, Springer Lecture Notes Series, 2021.

    \bibitem{Fur67}
    H. Furstenberg, Disjointness in ergodic theory, minimal sets, and a problem in Diophantine approximation, Math. Syst. Theory {\bf 1} (1967), 1--49.

    \bibitem{Hut81}
    J. E. Hutchinson, Fractals and self-similarity, Indiana Univ. Math. J. {\bf 30} (1981), 713--747.
    
    \bibitem{Mar54}
    J. M. Marstrand, Some fundamental geometrical properties of plane sets of fractional dimensions, Proc. Lond. Math. Soc. {\bf 4} (1954), 257--302.
    
    \bibitem{Mcm84}
    C. McMullen, The Hausdorff dimension of general Sierpi\'nski carpets, Nagoya Math. J. {\bf 96} (1984), 1--9.

    \bibitem{PV21}
    B. Paolillo and G. Vincenzi, On the rational values of trigonometric functions of angles that are rational in degrees, Mathematics Magazine {\bf 94} (2021), 132--134.

    \bibitem{PS09}
    Y. Peres and P. Shmerkin, Resonance between Cantor sets, Ergod. Th. \& Dynam. Sys. {\bf 29} (2009), 201--221. 

    \bibitem{RXtodedone}
    H. Rao and J.-C. Xiao, Relative interior of the intersection of self-similar sets, manuscript in preparation.


    \bibitem{SW99}
    R. Strichartz and Y. Wang, Geometry of self-affine tiles I, Indiana Univ. Math. J. {\bf 48} (1999), 1--23.


    \bibitem{Xiao24}
    J.-C. Xiao, On a self-embedding problem for self-similar sets, Ergod. Th. \& Dynam. Sys. {\bf 44} (2024), 3002--3011.

\end{thebibliography}

\end{document}